\documentclass[11pt]{article}
\usepackage{amsfonts}
\usepackage{amsmath}
\usepackage{amssymb}
\usepackage{amsxtra}
\usepackage{graphicx}
\usepackage{geometry}
\usepackage{color}
\usepackage{colortbl}
\usepackage{caption}%
\usepackage{amsmath,amssymb,tikz,graphicx,amsthm,comment, xcolor,tikz,mathtools,subfig}

\providecommand{\U}[1]{\protect\rule{.1in}{.1in}}
\newtheorem{theorem}{Theorem}[section]

\newtheorem{corollary}[theorem]{Corollary}

\newtheorem{definition}[theorem]{Definition}

\newtheorem{lemma}[theorem]{Lemma}

\newtheorem{remark}[theorem]{Remark}

\numberwithin{equation}{section}

\numberwithin{equation}{section}
\def\hb{\hbox}
\def\o {\Omega}

\def\ouno {\Omega_{1}}
\def\odue {\Omega_{2}}

\def\uuno {u_{1}}
\def\udue {u_{2}}

\def\nuno {n_{1}}
\def\ndue {n_{2}}

\def\e {\varepsilon}

\def\ve {V}
\def\hdue {H^{1}( \odue)}
\def\huno {H^{1}( \ouno)}
\def\ho#1 {H^{1}_0(#1)}
\def\luno {L^{2}(\ouno)}
\def\ldue {L^{2}(\odue)}

\def\lth2 {L^{2}(0,T;\;\hdue)}

\def\norq#1#2 {\| #1 \|_{#2}^{2}}
\def\nor#1#2 {\| #1 \|_{#2}}

\def\p#1#2 {P^{\e}_{#1}#2}

\def\g {\nabla}

\def\R{\mathbb{R}}

\begin{document}
\title{Exact Internal Controllability 
for a Problem with\\ Imperfect Interface}
\author{S. Monsurr\`{o}\thanks{Universit\`{a} di Salerno, Department of Mathematics, 84084 Fisciano (SA), Italy email: smonsurro@unisa.it},
A. K. Nandakumaran\thanks{Department of Mathematics, Indian Institute of Science, Bangalore 560012,  India; email: nands@iisc.ac.in},
C. Perugia\thanks{University of Sannio, Department of Science and Technology, Via de Sanctis, 82100 Benevento,
Italy; email: cperugia@unisannio.it}, }
\date{ }
\maketitle

\begin{abstract}
In this paper we study the internal exact controllability for a second order linear evolution equation defined in a two-component domain. On the interface we prescribe a jump of the solution proportional to the conormal derivatives, meanwhile a homogeneous Dirichlet condition is imposed on the exterior boundary. Due to the geometry of the domain, we apply controls through two regions  which are neighborhoods of a part of the external boundary and of the whole interface, respectively.
Our approach to internal exact controllability consists in proving an observability inequality by using the Lagrange multipliers method. Eventually we apply the Hilbert Uniqueness Method, introduced by J. -L. Lions, which leads to the construction of the exact control through the solution of an adjoint problem. Finally we find a lower bound for the control time depending not only on the geometry of our domain and on the matrix of coefficients of our problem but also on the coefficient of proportionality of the jump with respect to the conormal derivatives.
 \\
\\
\end{abstract}



\section{Introduction}\label{sec1}

In this paper, we plan to study the internal exact controllability of an imperfect transmission hyperbolic problem. More specifically, we consider a bounded domain $\Omega$ in  $\R^n$, $n\geq 2$, consisting of two sets   $\Omega_1$ and  $\Omega_2$, where $\Omega_2$ is compactly contained in $\Omega$ and  $\Omega_1= \Omega \setminus \overline{\Omega}_2$. Thus our domain has an external boundary $\partial\Omega$ and an  interface boundary $\Gamma$ (see Fig. 1).

\definecolor{c1}{RGB}{255,90,100}   
	\definecolor{c2}{RGB}{255,150,150}  
	\definecolor{c3}{RGB}{255,200,0}    
	\definecolor{c4}{RGB}{255,255,100}	

		\begin{center}
			\begin{tikzpicture}
				\filldraw[c2] plot [smooth cycle] coordinates {(2,0)(1,1)(0,2)(-2,1.6)(-2,0)(0,-2)(1,-1.6)};
				\filldraw[c4] plot [smooth cycle] coordinates {(1,0)(0.5,0.5)(0,1)(-1,0.8)(-1,0)(0,-1)(0.5,-0.8)};
				\draw(0,0)node{$\Omega_2$};
				\draw(-1,1.4)node{$\Omega_1$};
				\draw(1,1.5)node{$\Omega$};
				\draw(1,-2.5)node{Fig. 1: $\Omega_2\subset\subset\Omega$ and $\Omega_1=\Omega\backslash\overline\Omega_2$};
				\draw(0.9,-0.8)node{$\Gamma$};
				\draw(2,-0.8)node[right]{Interface Boundary};
				\draw[->](2,-0.8)--(1,-0.8);
				\draw(-1.6,-1)node{$\partial\Omega$};
			\end{tikzpicture}
		\end{center}

The hyperbolic  problem is defined in the set $\Omega$,  with appropriate interface and boundary conditions on $\Gamma$ and on $\partial\Omega$.
Namely, on the interface separating the two components, we prescribe a jump of the solution proportional to the conormal derivatives, meanwhile a homogeneous Dirichlet condition is imposed on the exterior boundary.  From the physical point of view, this problem describes the wave propagation in a composite made up of two materials with very different coefficients of propagation.
 The jump on $\Gamma$ is the mathematical interpretation of imperfect interface characterized by the discontinuity of the displacement (see \cite{1,3}, \cite{8}$\div$\cite{10}, \cite{FM, FMPhomo}, \cite{13,15,16}, \cite{17}$\div$\cite{MP1}, \cite{Zha1, Zha2} and references therein).
 
 The issue of exact controllability consists in acting on the trajectories of an evolution system by means of a control (the right hand side of the system, the boundary conditions, etc.) and asking if,
given a time interval $[0,T]$,   is it possible to find a control (or set of controls) leading the system to a desired state at time $T$, for all initial data. In a suitable functional setting, the problem of exact controllability reduces to that of observability. Roughly speaking, observability consists in deriving an estimate for the energy of an uncontrolled system, at time $t=0$,  in terms of partial measurements of its solution done on the control region. This estimate easily implies an upper bound for the norm of the initial data of the uncontrolled problem. 

The observability inequality, far from being obvious, forces the control set to satisfy suitable geometric conditions. Indeed, in \cite{BLR}, through the microlocal  approach, the authors proved  that when considering a regular domain, the observability inequality holds if and only if every ray of geometric optics, propagating into the domain and reflecting on its boundary, enters the control region in time less than the control time $T$.

In this paper, we do not require any regularity on $\partial \Omega$ and we make use of Lagrange multipliers method to prove the above mentioned observability inequality.
In general, when applying this technique, the assumptions on the geometry of the control region are very restrictive. They require the control set to be a neighborhood of parts of the boundary having specific structures.
In our case, due to the geometry of the domain, we need to introduce a further control set  which is a neighborhood of the interface $\Gamma$. More precisely, we apply controls through two regions $\omega_1 \subset \Omega_1$ and $\omega_2 \subset \Omega_2$ which are neighborhoods of a part of $\partial \Omega$  and of the whole interface $\Gamma$, respectively (see Fig. 2 and  Fig 3 for control region $\omega_1$). 
Other novelties in our framework are the jump of the solution on the interface $\Gamma$ and the resulting presence of a non-constant coefficients matrix. Due to the imperfect interface, when proving observability inequality, some difficulties arise in estimating  specific surface integrals.
Moreover, as usual in the hyperbolic framework, due to the finite speed of propagation of waves, the control time $T$ in our observability inequality has to be large enough. Indeed the control acting on $\omega_1$ and $\omega_2 $  cannot transfer the information immediately to the whole domain $\Omega$. However, unlike classical cases, we find a lower bound for the control time $T$ depending not only on the geometry of our domain and on the matrix of coefficients of our problem but also on the coefficient of proportionality of the jump with respect to the conormal derivatives.
The exact details are given in Section \ref{sec3}.

\bigskip 
	\begin{center}
			\begin{tikzpicture}[scale=0.1]
				\filldraw[c1] plot [smooth cycle] coordinates {(20,0)(10,10)(0,20)(-20,16)(-20,0)(0,-20)(10,-16)};
				\filldraw[c2,dashed] plot [smooth cycle] coordinates {(18,0)(9,9)(0,18)(-18,14.4)(-18,0)(0,-18)(9,-14.4)};
				\draw[dashed] plot [smooth cycle] coordinates {(18,0)(9,9)(0,18)(-18,14.4)(-18,0)(0,-18)(9,-14.4)};
				\filldraw[c3] plot [smooth cycle] coordinates {(10,0)(5,5)(0,10)(-10,8)(-10,0)(0,-10)(5,-8)};
				\filldraw[c4,dashed] plot [smooth cycle] coordinates {(8,0)(4,4)(0,8)(-8,6.4)(-8,0)(0,-8)(4,-6.4)};
				\draw[dashed] plot [smooth cycle] coordinates {(8,0)(4,4)(0,8)(-8,6.4)(-8,0)(0,-8)(4,-6.4)};
				\draw(10,15)node{$\Omega$};
				\draw(0,0)node{$\Omega_2$};
				\draw(-10,14)node{$\Omega_1$};
				\draw(-20,0)node[left]{control region $\omega_1$};
				\draw(-20,-15)node[left]{control region $\omega_2$};
				\draw[->,thick](-30,2)--(-21,10);
				\draw[->,thick](-30,-13)--(-10,5);
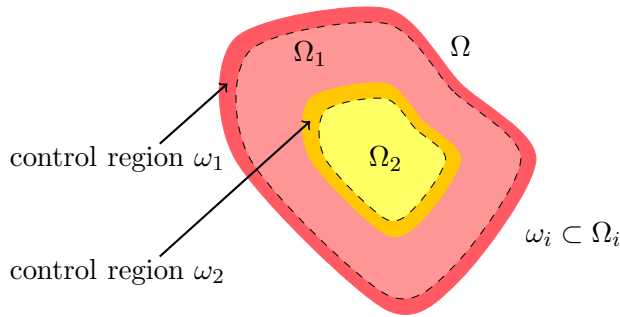
				\draw(0,-25)node{Fig. 2: $\omega_1$ -neighbourhood of $\partial\Omega$};
				\draw(0,-30)node{ \quad\quad\quad \quad \qquad$\omega_2$ -neighbourhood of $\Gamma=\partial\Omega_2$};
				\draw(25,-10)node{$\omega_i\subset\Omega_i$};
			\end{tikzpicture}
		\end{center}
		Once obtained the observability inequality, in order to find the exact control, we use a constructive method,  introduced by Lions in  \cite{ki27, ki27*}, known as Hilbert Uniqueness Method (HUM for short). The main feature is to build a control through the solution of a hyperbolic problem associated to suitable initial conditions. These initial conditions are obtained by calculating at zero time the solution of a backward problem by means of a functional which turns out to be an isomorphism, thanks to the observability estimate (see Section \ref{secHUM}).
Let us recall that the control obtained by HUM is also an energy minimizing control.\\

The paper is organized as follows. In Section 2, we introduce the setting of the evolution problem,  recall the definitions and some properties of the  appropriate functional spaces required for the solutions  of interface problems under consideration. For more details, we refer the reader to \cite{10} and \cite{17} where the elliptic case is considered. Since the initial data of our problem are in a weak space, the related solution cannot be defined using the standard weak formulation. Thus, as usual when dealing with controllability problems, we need to apply the so called transposition method (see \cite{ki28}, Chapter 3, Section 9). We also give the definition of exact controllability.

Section 3 is the core of the paper which is devoted to the proof of the observability inequality (see Lemma \ref{invineq5**}). To this aim, we adapt to our context some arguments introduced in \cite{ki27*} and \cite{ki27**}. By means of the Lagrange multipliers method, we derive an important identity (see Lemma \ref{lemfirstid}). Then, we specify the required geometrical and topological assumptions on the position of the observer and on the control sets $\omega_1$ and $\omega_2$ (see  Definitions \ref{omega1} and \ref{omega2}) and apply the above mentioned identity in order to establish some crucial inequalities, given in Lemmas \ref{invineq1}, \ref{invineq2} and \ref{invineq3}. 
Finally, in Lemma \ref{invineq4}, we find the lower bound for the control time $T$. Taking into account the way $\omega_1$ is constructed, it is possible to get different regions of controllability depending on the observer point $x^0$. The significance of the point $x^0$ will be clear in Section 3.
In Section \ref{secHUM}, via HUM, we prove the exact controllability result by constructing the suitable isomorphism that allows us to identify the exact control.

The pioneer studies on exact controllability for the wave equation with transmission conditions, via HUM, go back to \cite{ki27*}, Chapter $6$. Here J. L. Lions  considers a Dirichlet problem with matrix constant on each component of the domain and a control on part of the external boundary. Later on, in \cite{LW1} the authors deal with the case of a Neumann boundary value problem in the same framework. In \cite{ANP1}-\cite{ANP4}, \cite{DGL}-\cite{DNP}, \cite{Durante1}-\cite{D3}, \cite{NS}-\cite{NRP} optimal control and exact controllability problems in domains with highly oscillating boundary are studied. Moreover we refer to \cite{ki27} and \cite{ CD1, ki15,ki17} for the exact controllability of hyperbolic problems with oscillating coefficients in fixed and in perforated domains respectively, to \cite{FP2,FP3} and  \cite{FMPhomo,FMP, MP1} for the optimal control and the exact controllability, respectively, of hyperbolic problems in composites with imperfect interface. Moreover in \cite{FP4} it is faced the optimal control of rigidity parameters of thin inclusions in composite materials.
In \cite{DJ} and \cite{DJ1, DJ2} the authors study, respectively, the correctors and the approximate control for a class of parabolic equations with interfacial contact resistance, while in \cite{DoNa} the authors study the approximate controllability of linear parabolic equations in perforated domains. In \cite{Z2002}, see also \cite{Z94}, the author studies the approximate controllability of a parabolic problem with highly oscillating coefficients in a fixed domain.
The null controllability of semilinear heat equations in a fixed domain was done in \cite{FC}. The exact controllability and exact boundary controllability for semilinear wave equations can be found in \cite{LZ} and \cite{Z91}, respectively.
Finally, for what concerns transmission problems in the nonlinear case we quote \cite{DR1} and \cite{DR2} (see also \cite{DMR1} and \cite{DMR2}).
 	
\section{Statement of the problem}\label{sec2}
Let  $\o$ be a connected open bounded subset of $\R^n$, $n\geq 2$. We denote by $\o_1$ and $\o_2$, two non-empty open connected and disjoint subsets of $\o$
such that $\overline{\Omega}_2 \subset \o$ and $\ouno  = \o \setminus\odue$.
Let us assume that the interface $\Gamma=\partial {\o_{2}}$ separating the two components of $\o$  is  Lipschitz continuous and  observe that by construction one has
\begin{equation}\label{eq2.1}
\partial\Omega\cap \Gamma =\emptyset.
\end{equation}
Given $T>0$, we set $Q_1=\ouno\times (0,T)$, $Q_2=\odue\times (0,T)$, $\Sigma=\partial \o \times (0,T)$ and $\Gamma_T=\Gamma \times (0,T)$. 
\\
This paper aims to study the internal exact controllability of a hyperbolic imperfect transmission problem defined in the above mentioned  domain. 
More precisely, given two open subsets $\omega_i$ of $\Omega_i$, $i=1,2$, and a control $\zeta:=(\zeta_1, \zeta_2)$, we consider the problem
\begin{equation}\label{eq2.6}
\left\{
\begin{array}{@{}ll}
\uuno'' - \hbox {div} \left(A(x)
\nabla \uuno \right)=
    \zeta_1 \chi_{\omega_1} &\hbox {in } Q_1, \\[3pt]
\udue''- \hbox{div} \left(A(x) \nabla \udue \right) = \zeta_2 \chi_{\omega_2} &\hbox {in } Q_2, \\[3pt]
A(x) \nabla \uuno   \nuno =
   -A(x) \nabla \udue   \ndue& \hbox {on }
\Gamma_T, \\[3pt]
A(x)  \nabla \uuno   \nuno
   = -h(x) (\uuno-\udue)& \hbox {on }
\Gamma_T, \\[3pt]
\uuno = 0&\hbox {on } \Sigma,\\[3pt]
\uuno(0)= U_{1}^0,\quad \uuno'(0)= U^1_{1}&\hbox {in }  \ouno, \\[3pt]
 \udue(0)
= U_{2}^0, \quad \udue'(0)
= U^1_{2}&\hbox {in }  \odue,
\end{array}
\right.
\end{equation}
where for any fixed $i=1,2$, $n_{i}$ is the unitary outward normal to $ 
\Omega_{i}$ and $\chi_{\omega_i}$ denotes the characteristic function of the set $\omega_i$ which is the action of the control $\zeta_i$. For simplicity, we denote the wave operator by $L = \partial_{tt} -  \hbox {div} \left(A(x)
\nabla \>\right).$

Let us recall the definitions of the required  function spaces to study the interface problem under consideration. They were introduced for the first time in \cite{17} and successively in \cite{10} in the framework of the study of the homogenization of the analogous stationary problem. Indeed, these spaces take into account the geometry of the domain as well as the boundary and interface conditions. 
\\
As observed in \cite{7}, the space
$$ V=\  \{ v_1 \in H^1\left( \o_1 \right) |\;v_1=0\, \text{on}\, \partial\Omega\}$$
 is a Banach space endowed with the norm 
$$\| v_1 \|_{V}=\ \| \g v_1 \|_{\luno}.$$
Since we do not impose any regularity on the external boundary, the condition on $\partial \o$ in the definition of $\ve$ has to be
intended in a density sense. More precisely, in view of (\ref{eq2.1}), $\ve$ can be defined as the
closure of the set of the
functions in $C^{\infty}(\ouno)$ with a compact support contained
in $\o$ with respect to the $\huno$-norm.
We also set
\begin{equation}\label{heg}
H_{\Gamma}=\  \left\{v=
(v_1,v_2) \mid v_1 \in V \; \hbox {and} \; v_2 \in H^{1}(\odue)\right\}.
\end{equation} 
The space  $H_{\Gamma}$ is a Hilbert space when
equipped with the norm 
$$\| v \|^2_{H_{\Gamma}} =\ \| \g v_1 \|^2 _{\luno} + \| \g v_2 \|^2_{\ldue} 
+ \| v_1 -  v_2 \|^2 _{L^2(\Gamma)}. $$ 
Indeed $H_{\Gamma}$ can be identified with  $V
\times H^{1}( \odue)$, the norm above defined on $H_{\Gamma}$ being equivalent to the standard norm in $V
\times H^{1}( \odue)$ (see \cite{9} for details). We denote the dual of $H_{\Gamma}$ by $(H_{\Gamma})'$. It follows that (see \cite{DJ}),  the norms of $(H_{\Gamma})'$ and $V'\times (H^1(\Omega_2))'$ are equivalent. Moreover, if $v=(v_1,v_2)\in (H_{\Gamma})'$ and $u=(u_1,u_2)\in H_{\Gamma}$, then
$$
\left\langle v,u\right\rangle_{(H_{\Gamma})',H_{\Gamma}}=\left\langle v_1,u_1\right\rangle_{V',V}+\left\langle v_2,u_2\right\rangle_{H^1(\Omega_{2})',H^1(\Omega_{2})}.
$$

\begin{remark}\label{rem0}
We point out that $H_{\Gamma}$ is a separable and reflexive Hilbert space dense in $   L^{2}\left(\Omega_{1}\right)\times L^{2}\left(\Omega_{2}\right)$. Furthermore, $H_\Gamma \subseteq   L^{2}\left(\Omega_{1}\right)\times L^{2}\left(\Omega_{2}\right)$ with continuous imbedding. On the other hand, one has that
$  L^{2}\left(\Omega_{1}\right)\times L^{2}\left(\Omega_{2}\right)\subseteq\left(H_{\Gamma}\right)'$, with $  L^{2}\left(\Omega_{1}\right)\times L^{2}\left(\Omega_{2}\right)$ separable Hilbert space.
This means that the triple $(H_\Gamma,   L^{2}\left(\Omega_{1}\right)\times L^{2}\left(\Omega_{2}\right), \left(H_{\Gamma}\right)')$ is an evolution triple. We refer the reader to \cite{9,DFMC} for an in-depth analysis on this aspect. Also note that, in fact $L^{2}\left(\Omega_{1}\right)\times L^{2}\left(\Omega_{2}\right)$ can be identified with $L^{2}\left(\Omega\right)$ itself by observing that $v=(v_1,v_2)\in L^{2}\left(\Omega_{1}\right)\times L^{2}\left(\Omega_{2}\right)$ if and only if $v= v_1\chi_{\Omega_1} + v_2\chi_{\Omega_2} \in L^{2}\left(\Omega\right)$. By the way, due to the nature of our problem, throughout this work we prefer to adopt the notation  $v=(v_1,v_2)\in L^{2}\left(\Omega_{1}\right)\times L^{2}\left(\Omega_{2}\right)$.
\end{remark}

\noindent Let us set
\begin{equation}\label{W}
 W =\left\lbrace v=\left(v_{1},v_{2}\right)\in L^{2}\left(0,T;V\times H^{1}\left(\Omega_{2}\right)\right):\,
 v'=\left(v'_1,v'_2\right) \in L^{2}\left(0,T;  L^{2}\left(\Omega_{1}\right)\times L^{2}\left(\Omega_{2}\right)\right)\right\rbrace,
\end{equation}
\noindent
 which is a Hilbert space if equipped with the norm
$$\left\|v\right\|_{W}=\left\|v_{1}\right\|_{L^{2}\left(0,T;V\right)}+\left\|v_{2}\right\|_{L^{2}\left(0,T;H^{1}\left(\Omega_{2}\right)\right)}+\left\|v^{\prime}_{1}\right\|_{L^{2}\left(0,T;L^{2}\left(\Omega_{1}\right)\right)}+\left\|v^{\prime}_{2}\right\|_{L^{2}\left(0,T;L^{2}\left(\Omega_{2}\right)\right)}.$$
We assume that the initial data of problem \eqref{eq2.6} are such that
\begin{equation}\label{eq2.7}
\left\{
\begin{array}{@{}l}
\mbox{(i) }
U^{0} =\left( U^0_{ 1},  U^0_{ 2 }\right)\in   L^{2}\left(\Omega_{1}\right)\times L^{2}\left(\Omega_{2}\right),\\[3mm]
\mbox{(ii) }
U^{1} =\left( U^1_{ 1},  U^1_{ 2 }\right)\in\left(H_{\Gamma}\right)'
\end{array}\right.
\end{equation}
and that the control $\zeta$ is such that
\begin{equation}\label{ipozeta}
\zeta \in W'.
\end{equation}

\noindent Further, we suppose that $A$ is a symmetric  matrix field and there exist constants $\alpha,\beta\in\R$, with $0<\alpha<\beta$  such that
\begin{equation}\label{ipoA}
\left\{
\begin{array}{@{}l}
(i)\,\,a_{ij},\,\dfrac{a_{ij}}{\partial x_k}\in L^{\infty}\left(\Omega\right),\,\,\,1\leq i,j,k\leq n,\\[3mm]
(ii)\,\,(A(x)\lambda,\lambda)\geq \alpha {|\lambda|^{2}},\,
 |A(x) \lambda|\leq\beta|\lambda|,
\end{array}
\right.
\end{equation}
for every $\lambda\in\R^n $ and a.e. in $\Omega$. 
We put
\begin{equation}\label{defm}
M=\max_{1\leq i,j,k\leq n}\,\max_{x\in\Omega} \left|\dfrac{\partial a_{ ij }}{\partial x_k}\right|.
\end{equation}
\\
The function $h$ appearing in the interface condition satisfies
\begin{equation}\label{ipoh}
h \in
L^{\infty}(\Gamma) \;\hbox{and there exists} \; h_0 \in \R \;
\hb{ such that }
   0<h_0<h(x) \>\text{   a.e.\ in  } \;\Gamma.
   \end{equation}
   
\noindent    Note that  the initial data \eqref{eq2.7} are in a weak space, hence the solution of problem  \eqref{eq2.6} cannot be defined using the standard weak formulation. We need to apply the so called transposition method (see \cite{ki28}, Chapter 3, Section 9) usually used in controllability problems. In some sense, it is an adjoint method where the solution is defined via an adjoint problem which provides test functions. 
More precisely, we define the following standard adjoint  problem: for every $g =(g_{1},g_{2})\in L^2(0,T;  L^{2}\left(\Omega_{1}\right)\times L^{2}\left(\Omega_{2}\right))$,  consider the backward problem
\begin{equation}\label{backprobeps}
\left\{
\begin{array}{@{}ll}
L\psi_i \equiv \psi_{i}'' - \hbox {div} (A(x)\nabla \psi_{i} )=g_{i} & \hbox {in } Q_i, \; i=1,2\\[3pt]
A(x) \nabla \psi_{1} \;  \nuno =
   -A(x) \nabla \psi_{2} \;    \ndue & \hbox {on }
\Gamma_T, \\[3pt]
A(x) \nabla \psi_{1}  \;   \nuno
   = -h(x)(\psi_{1}-\psi_{2}) & \hbox {on }
\Gamma_T, \\[3pt]
\psi_{1} = 0 & \hbox {on } \Sigma,\\[3pt]
\psi_{i}(T)=\psi'_{i}(T)
= 0 & \hbox {in }  \Omega_{i},  \; i=1,2.

\end{array}
\right.
\end{equation}

As observed in \cite{9}, thanks to Remark \ref{rem0}, by using an approach to standard evolutionary problems based on evolution triples (there are no weak data),  the usual weak formulation of problem \eqref{backprobeps} is valid. Hence an abstract Galerkin's method provides the existence and uniqueness result for the weak solution  in $W$ of problem \eqref{backprobeps}, together with the a priori estimate for the solution in $W$.
For the sake of clarity, throughout the paper, we denote by $\psi(g) =(\psi_{1}(g),\psi_{2}(g))$, the solution of problem \eqref{backprobeps} and when there is no ambiguity, we omit the explicit dependence on the right hand member.

 \noindent Now we give the definition of solution of \eqref{eq2.6} in the sense of transposition.
\begin{definition}\label{deftraspeps}
For any fixed $\left(U^{0}, U^{1}\right)\in   \left(L^{2}\left(\Omega_{1}\right)\times L^{2}\left(\Omega_{2}\right)\right)\times \left(H_{\Gamma}\right)'$, we say that a function $u =(u_{1},u_{2})\in L^2(0,T;   L^{2}\left(\Omega_{1}\right)\times L^{2}\left(\Omega_{2}\right))$ is a solution of problem \eqref{eq2.6}, in the sense of transposition, if it satisfies the identity
\begin{equation}\label{soltras}
\begin{array}{c}
\displaystyle\int_{Q_1}u_{1}g_1dxdt+
\displaystyle\int_{Q_2}u_{2}g_2 dxdt =-\displaystyle\int_{\ouno}U^{0}_{1}\psi'_{1}(0)dx+
\left\langle U^{1}_{1},\psi_{1}(0)\right\rangle_{V',V}\\
\\ \qquad\qquad\qquad-\displaystyle\int_{\odue}U^{0}_{2}\psi'_{2}(0)dx
+\left\langle U^{1}_{2},\psi_{2}(0)\right\rangle_{(H^1(\odue))',H^1(\odue)}
+\left\langle \zeta \chi_{\omega}, \psi \right\rangle_{W',W},\end{array}
\end{equation}
for all $g =(g_{1},g_{2})\in L^2(0,T;  L^{2}\left(\Omega_{1}\right)\times L^{2}\left(\Omega_{2}\right))$, where $\psi$ is the solution of problem \eqref{backprobeps}. Here we have used the notation  $\zeta \chi_{\omega} =(\zeta_1 \chi_{\omega_1}, \zeta_2 \chi_{\omega_2})$.
\end{definition}
\noindent By classical results (see \cite{ki28}, Chapter 3, Section 9, Theorems 9.3 and 9.4), problem \eqref{eq2.6} admits a unique solution $u\in C\left([0,T];  L^{2}\left(\Omega_{1}\right)\times L^{2}\left(\Omega_{2}\right)\right)\cap C^{1}\left([0,T];\left(H_{\Gamma}\right)'\right)$ satisfying the estimate
 \begin{equation}\label{stimatrasp1}
 \|u\|_{L^\infty(0,T;   L^{2}\left(\Omega_{1}\right)\times L^{2}\left(\Omega_{2}\right))}+ \|u'\|_{L^\infty(0,T; (H_{\Gamma})') }\leq C( \|U^0\|_{  L^{2}\left(\Omega_{1}\right)\times L^{2}\left(\Omega_{2}\right)}+ \|U^1\|_{(H_{\Gamma})' }+ \|\zeta \chi_{\omega}\|_{W'}),
 \end{equation}
with $C$ positive constant.\\
We denote by $u\left(\zeta\right) =\left(\uuno\left(\zeta\right),\udue\left(\zeta\right)\right)$ the solution of problem \eqref{eq2.6} in the sense above defined and, when there is no ambiguity, we omit the explicit dependence on the control.  \\
Now, let us give the definition of exact controllability.
\begin{definition}
System \eqref{eq2.6} is said to be exactly controllable at time $T>0$, if for every $\left(U^{0},U^{1}\right)$, $\left(Z^{0},Z^{1}\right)$ in $\left(L^{2}\left(\Omega_{1}\right)\times L^{2}\left(\Omega_{2}\right)\right) \times (H_\Gamma)'$, there exists a control $\zeta$ belonging to $W'$ such that the corresponding solution $u$ of problem \eqref{eq2.6} satisfies
$$u(T)=Z^{0},\,\,\,u'(T)=Z^{1}.$$
\end{definition}

If the controllability is achieved for the zero (null) data $Z^0=0, Z^1 =0$, then it  is known as null controllability. Since our problem is linear, it is sufficient to look for controls driving the system
 to rest. Hence, in the sequel we prove the existence of a control $\zeta\in W'$ of  \eqref{eq2.6} such that $u(T)=u'(T)=0.$

\smallskip

In this paper, we wish to prove a controllability result, but as remarked in the introduction, it is not possible to achieve controllability without additional assumptions. Namely, 
if $\Omega_2$ is star-shaped with respect to a point $x^0 \in \Omega_2$ and under suitable geometrical  assumptions on the sets  $\omega_1$ and $\omega_2$, we are able to prove that system \eqref{eq2.6} is exact controllable for a time $T > 0$ sufficiently large (see Theorem \ref{mainteo}). To this aim, we will use a constructive method known as the Hilbert Uniqueness Method introduced by Lions (see \cite{ki27, ki27*}). We point out that the control obtained by HUM is also the energy minimizing control. HUM is fully a PDE based method and eventually it reduces to deriving the so-called observability estimate, which is the crucial point, corresponding to an uncontrolled problem (see \eqref{eq2.6h}). To get  the observability estimate, which is a delicate estimate from below, we need to establish some fundamental results based on the  Lagrange multipliers method. These results are proved in the following section.

\section{The observability inequality}\label{sec3}
For $T>0$, we consider the following homogeneous imperfect transmission problem 
\begin{equation}\label{eq2.6h}
\left\{
\begin{array}{@{}ll}
Lz_i\equiv z_i'' - \hbox {div} \left(A(x)
\nabla z_i \right)=
   0&\hbox {in } Q_i, \; i=1,2 \\[3pt]
A(x) \nabla z_1 \; \nuno =
   -A(x) \nabla z_2\;  \ndue& \hbox {on }
\Sigma_\Gamma, \\[3pt]
A(x)  \nabla  z_1  \; \nuno
   = -h(x) (z_1-z_2)& \hbox {on }
\Sigma_\Gamma, \\[3pt]
z_1 = 0&\hbox {on } \Sigma,\\[3pt]
z_i(0)= z_{i}^0,\quad z_i'(0)= z^1_{i}&\hbox {in }  \Omega_i,  \; i=1,2\\
\end{array}
\right.
\end{equation}
with the initial data
\begin{equation}\label{regip}
z^{0} =\left( z^0_{ 1},  z^0_{ 2 }\right)\in H_\Gamma, \quad
z^{1} =\left( z^1_{ 1},  z^1_{ 2 }\right)\in L^{2}\left(\Omega_{1}\right)\times L^{2}\left(\Omega_{2}\right),
\end{equation}
where $n_{i}$ is the unitary outward normal to $ 
\Omega_{i}$,  $i=1,2$.
The weak formulation of problem \eqref{eq2.6h} is given by
\begin{equation}\label{ipervar0}
\left\{
\begin{array}{@{}l}
\hbox{Find} \; z = ( z_{1}, z_{2}) \; \hbox {in} \,W  \hb{ such that }\\[3pt]
\langle z_{1}'', v_1 \rangle_{V', V} + \langle z_{2}'', v_2 \rangle_{(\hdue)',\hdue}
\displaystyle  + \int_{\ouno}{A(x) \nabla z_{1}  \nabla v_{1}}\; dx
 + \int_{\odue}{A(x) \nabla z_{2}  \nabla v_{2}} \; dx \\[9pt] \qquad + \displaystyle\int_{\Gamma}{h(x)(z_{1} - z_{2})(v_1 - v_2) \; d\sigma_x}=
    0 \;
\hb{ for all}\; (v_1,v_2) \in \ve \times \hdue \hb{ in }{\cal D}'(0,T),\\[3pt]
\\
z_i(0)= z_{i}^0,\quad z_i'(0)= z^1_{i} \hbox { in }  \Omega_i,  \; i=1,2.
\end{array} \right.
\end{equation}
As already observed, in \cite{9} the authors prove  the existence and uniqueness result for the weak solution in $W$ of problem \eqref{eq2.6h} together with some a priori estimates. 
\begin{theorem}[\cite{9}]
\label{3.1} Let $T>0$, $H_{\Gamma}$ and $W$ be defined as in \eqref{heg} and 
$(\ref{W})$. Under hypotheses $(\ref{ipoA}), \>(\ref{ipoh}) $ and  \eqref{regip}, problem $(\ref{eq2.6h})$ admits a unique weak solution $z\in W$. Moreover,
there exists a positive constant $C$, such that
\begin{equation}\label{stima}
\left\Vert z\right\Vert _{L^{\infty}(0,T;H_{\Gamma})}+\left\Vert z'\right\Vert
_{L^{\infty}(0,T; L^{2}\left(\Omega_{1}\right)\times L^{2}\left(\Omega_{2}\right))}\leq
 C\left(
\left\Vert z^{0}\right\Vert _{H_{\Gamma}
}+\left\Vert z^{1}\right\Vert _{ L^{2}\left(\Omega_{1}\right)\times L^{2}\left(\Omega_{2}\right)}\right).
\end{equation}

\end{theorem}

\noindent Let us remark that, the solution of
problem \eqref{eq2.6h} has some further properties (see $\cite{ki28}$,
Chapter 3, Theorem 8.2). In fact, under the same
hypotheses of Theorem $\ref{3.1}$, the unique solution $z$ of
problem \eqref{eq2.6h} satisfies
\begin{equation*}
z\in C\left(  \left[  0,T\right]  ;H_{\Gamma}\right),\,
z'\in C\left(  \left[  0,T\right]  ;  L^{2}\left(\Omega_{1}\right)\times L^{2}\left(\Omega_{2}\right) \right).
\end{equation*}
\noindent Hence the initial values $z(0)$ and $z'(0)$ are meaningful in the appropriate spaces. 

Now, we derive an important identity using suitable multipliers. It is essential for establishing the inverse inequalities involved in the exact controllability problem. For convenience, we use the repeated index summation convention. Moreover, when there is no ambiguity, we omit the explicit dependence on the  space variable $x$ in the matrix $A$ and  in the function $h$.
\begin{lemma}\label{lemfirstid}
Let $q=(q_1, \ldots q_n)$ be a vector field in $(W^{1,\infty}(\Omega))^n$ and let $z=(z_1,z_2)$ be the solution of problem \eqref{eq2.6h}-\eqref{regip}. Then, the following identity holds
\begin{equation}\label{firstid}
\begin{array}{c}
\displaystyle\dfrac{1}{2}\int_\Sigma A n_1  n_1\left(\dfrac{\partial z_1}{\partial n_1}\right)^2q_k n_{1k}\;d\sigma_x\;dt+\dfrac{1}{2}\sum_{i=1}^{2}\int_{\Gamma_{T}} A n_i  n_i\left(\dfrac{\partial z_i}{\partial n_i}\right)^2q_k n_{ik}\;d\sigma_x\;dt+\\
\\
-\displaystyle\int_{\Gamma_T} h \left(z_1-z_2\right)q_k\left(\nabla_\sigma (z_1-z_2)\right)_k\;d\sigma_x\;dt+ \dfrac{1}{2}\sum_{i=1}^{2}\int_{\Gamma_T} \left(|z'_i|^2-A \nabla_\sigma z_i\nabla_\sigma z_i\right)q_kn_{ik}\;d\sigma_x\;dt=\\
\\
=\displaystyle\sum_{i=1}^2\left.\left(z'_i,q_k\dfrac{\partial z_i}{\partial x_k}\right)_{\Omega_i}\right |_{0}^{T} +\dfrac{1}{2}\sum_{i=1}^{2}\int_{Q_i} \left(|z'_i|^2-A \nabla z_i\nabla z_i\right)\dfrac{\partial q_k}{\partial x_k}\;dx\;dt+\\
\\
\qquad\qquad+\displaystyle\sum_{i=1}^{2}\int_{Q_i} A \nabla z_i\nabla q_k\dfrac{\partial z_i}{\partial x_k}\;dx\;dt-\dfrac{1}{2}\sum_{i=1}^{2}\int_{Q_i} q_k
\sum_{l,j=1}^{n}\dfrac{\partial a_{lj}}{\partial x_k}\dfrac{\partial z_i}{\partial x_l}\dfrac{\partial z_i}{\partial x_j}\;dx\;dt,
\end{array}
\end{equation}
where
$$
\left(z'_i,q_k\dfrac{\partial z_i}{\partial x_k}\right)_{\Omega_i}=\int_{\Omega_i}z'_i(t)q_k\dfrac{\partial z_i(t)}{\partial x_k}\,dx
$$
and  $\nabla_\sigma z_i=(\sigma_j z_i)_{j=1}^n$ denotes the tangential gradient of $z_i$ on $\Gamma$ for $i=1,2$ (see, for instance,  \cite{ki27*},  p. 137]).
\end{lemma}

\begin{proof}
We prove the result for a strong solution of problem \eqref{eq2.6h}, that is under the following more regular initial data \begin{equation}\label{regin}
\left\{
\begin{array}{@{}l}

z^{0} =\left( z^0_{ 1},  z^0_{ 2 }\right)\in (H^2(\Omega_1)\cap V)\times H^2(\Omega_2),\\[3pt]

z^{1} =\left( z^1_{ 1},  z^1_{ 2 }\right)\in V\times H^{1}\left(\Omega_2\right).
\end{array}\right.
\end{equation}
Indeed one can easily prove that it holds also considering the weaker hypotheses \eqref{regip} (see for instance \cite{ki27*}). 

Let us multiply the first equation in \eqref{eq2.6h} by $q_k\dfrac{\partial z_1}{\partial x_k}$ and then integrate on $[0,T]$ to get

\begin{equation}\label{fid1}
\begin{array}{l}
\displaystyle \int_0^T\langle z_{1}'', q_k\dfrac{\partial z_1}{\partial x_k} \rangle_{V', V}\;dt +  \int_{Q_1}A \nabla z_{1}  \nabla \left(q_k\dfrac{\partial z_1}{\partial x_k}\right)\; dx\;dt\\
\\
\qquad \qquad\qquad-\displaystyle\int_{\Sigma}A\nabla z_1q_k\dfrac{\partial z_1}{\partial x_k}  n_1\; d\sigma_x\;dt-\int_{\Gamma_T}A\nabla z_1q_k\dfrac{\partial z_1}{\partial x_k}  n_1\; d\sigma_x\;dt=
  0. 
\end{array}
\end{equation}
For clearness sake, let us rewrite the above identity as
\[I_1 +I_2 + I_3 + I_4 = 0.\]
Applying integration by parts and Gauss-Green theorem repeatedly, we get 
\begin{equation*}
\begin{array}{ll}
I_1 &= \displaystyle \left.\int_{\Omega_1}z_{1}' q_k\dfrac{\partial z_1}{\partial x_k}\;dx \right|_{0}^T - \int_{Q_1}z'_1q_k\dfrac{\partial z'_1}{\partial x_k} \;dx\;dt\\
\\
&= \displaystyle \left.\int_{\Omega_1}z_{1}' q_k\dfrac{\partial z_1}{\partial x_k}\;dx \right|_{0}^T -\dfrac{1}{2} \int_{Q_1}\dfrac{\partial}{\partial x_k} \left|z'_1\right|^2 q_k\;dx\;dt \\
\\
&=\displaystyle \left.\int_{\Omega_1}z_{1}' q_k\dfrac{\partial z_1}{\partial x_k}\;dx \right|_{0}^T +\dfrac{1}{2} \int_{Q_1}\left|z'_1\right|^2\dfrac{\partial q_k}{\partial x_k}\;dx\;dt
\displaystyle-\dfrac{1}{2}\int_\Sigma \left|z'_1\right|^2q_k n_{1k}\;d\sigma_ x\;dt\\
\\
&\quad\qquad
\displaystyle -\dfrac{1}{2}\int_{\Gamma_T} \left|z'_1\right|^2q_k n_{1k}\;d\sigma_ x\;dt.
\end{array}
\end{equation*}
Since  $z_1=0$ on $\Sigma$ implies $z'_1=0$  on $\Sigma$ by stronger regularity assumptions, the third term vanishes. Hence we get
\[ I_1 = \displaystyle \left.\int_{\Omega_1}z_{1}' q_k\dfrac{\partial z_1}{\partial x_k}\;dx \right|_{0}^T +\dfrac{1}{2} \int_{Q_1}\left|z'_1\right|^2\dfrac{\partial q_k}{\partial x_k}\;dx\;dt -\dfrac{1}{2}\int_{\Gamma_T} \left|z'_1\right|^2q_k n_{1k}\;d\sigma_ x\;dt.\]

Now, we compute $I_2$
\begin{equation*}
\begin{array}{ll}
I_2 &= \displaystyle \int_{Q_1}A \nabla z_{1}  \nabla q_k\dfrac{\partial z_1}{\partial x_k}\; dx\;dt+\int_{Q_1}A \nabla z_{1}  q_k\nabla \dfrac{\partial z_1}{\partial x_k}\; dx\;dt\\
\\
&= \displaystyle \int_{Q_1}A \nabla z_{1}  \nabla q_k\dfrac{\partial z_1}{\partial x_k}\; dx\;dt+
\int_{Q_1}A \nabla z_{1}  q_k\dfrac{\partial}{\partial x_k}\nabla  z_1\; dx\;dt \\
\\
&= 
\displaystyle\int_{Q_1}A \nabla z_{1}  \nabla q_k\dfrac{\partial z_1}{\partial x_k}\; dx\;dt-\dfrac{1}{2}\int_{Q_1}A\nabla z_1\nabla z_1 \dfrac{\partial q_k}{\partial x_k} \; dx\;dt\displaystyle - \dfrac{1}{2}\int_{Q_1}q_k\sum_{l,j=1}^n \dfrac{\partial a_{lj}}{\partial x_k} \dfrac{\partial z_1}{\partial x_l} \dfrac{\partial z_1}{\partial x_j}\; dx\;dt\\
\\
&\qquad\qquad 
\displaystyle 
+\displaystyle \dfrac{1}{2}\int_\Sigma A\nabla z_1\nabla z_1 q_kn_{1k}\;d\sigma_x\;dt+
\dfrac{1}{2}\int_{\Gamma_T} A\nabla z_1\nabla z_1 q_kn_{1k}\;d\sigma_x\;dt.
\end{array}
\end{equation*}

Moreover, since $z_1=0$ on $\Sigma$, one has  $\nabla z_1=\dfrac{\partial z_1}{\partial n_1} n_1$ on $\Sigma$, that is $\dfrac{\partial z_1}{\partial x_k} = \dfrac{\partial z_1}{\partial n_1} n_{1k}$, hence $I_3$ becomes 
\[ I_3 =  -\displaystyle\int_{\Sigma}A n_1n_1q_k n_{1k}\left(\dfrac{\partial z_1}{\partial n_1}\right)^2\; d\sigma_x\;dt.\]
%
Also note that the fourth term  in the last expression for $I_2$ is $-\frac{1}{2} I_3$. Combining the computations for $I_1, I_2, I_3$ in \eqref{fid1}, we can get the following identity for $z_1$
\begin{equation}\label{fid3'}
\begin{array}{c}
\displaystyle \left.\int_{\Omega_1}z_{1}' q_k\dfrac{\partial z_1}{\partial x_k}\;dx \right|_{0}^T +\dfrac{1}{2} \int_{Q_1}\left|z'_1\right|^2\dfrac{\partial q_k}{\partial x_k}\;dx\;dt -\dfrac{1}{2}\int_{\Gamma_T} \left|z'_1\right|^2q_k n_{1k}\;d\sigma_ x\;dt\\
\\
+
\displaystyle\int_{Q_1}A \nabla z_{1}  \nabla q_k\dfrac{\partial z_1}{\partial x_k}\; dx\;dt-\dfrac{1}{2}\int_{Q_1}A\nabla z_1\nabla z_1 \dfrac{\partial q_k}{\partial x_k} \; dx\;dt\\
\\
\displaystyle -\dfrac{1}{2}\int_{Q_1}q_k\sum_{l,j=1}^n \dfrac{\partial a_{l,j}}{\partial x_k} \dfrac{\partial z_1}{\partial x_l} \dfrac{\partial z_1}{\partial x_j}\; dx\;dt
+\displaystyle \dfrac{1}{2}\int_\Sigma A\nabla z_1\nabla z_1 q_kn_{1k}\;d\sigma_x\;dt+
\dfrac{1}{2}\int_{\Gamma_T} A\nabla z_1\nabla z_1 q_kn_{1k}\;d\sigma_x\;dt\\
\\
-\displaystyle\int_{\Sigma}A n_1n_1q_k n_{1k}\left(\dfrac{\partial z_1}{\partial n_1}\right)^2\; d\sigma_x\;dt-\int_{\Gamma_T}A\nabla z_1  n_1 q_k\dfrac{\partial z_1}{\partial x_k}\; d\sigma_x\;dt=
    0. 
\end{array}
\end{equation}
Analogously, multiplying the second equation in \eqref{eq2.6h} by $q_k\dfrac{\partial z_2}{\partial x_k}$ and then integrating on $[0,T]$, we get
\begin{equation}\label{fid4'}
\begin{array}{c}
\displaystyle \left.\int_{\Omega_2}z_{2}' q_k\dfrac{\partial z_2}{\partial x_k}\;dx \right|_{0}^T +\dfrac{1}{2} \int_{Q_2}\left|z'_2\right|^2\dfrac{\partial q_k}{\partial x_k}\;dx\;dt -\dfrac{1}{2}\int_{\Gamma_T} \left|z'_2\right|^2q_k n_{2k}\;d\sigma_ x\;dt\\
\\
+
\displaystyle\int_{Q_2}A \nabla z_{2}  \nabla q_k\dfrac{\partial z_2}{\partial x_k}\; dx\;dt-\dfrac{1}{2}\int_{Q_2}A\nabla z_2\nabla z_2 \dfrac{\partial q_k}{\partial x_k} \; dx\;dt\\
\\
\displaystyle-\dfrac{1}{2}\int_{Q_2}q_k\sum_{l,j=1}^n \dfrac{\partial a_{l,j}}{\partial x_k} \dfrac{\partial z_2}{\partial x_l} \dfrac{\partial z_2}{\partial x_j}\; dx\;dt+
\dfrac{1}{2}\int_{\Gamma_T} A\nabla z_2\nabla z_2 q_kn_{2k}\;d\sigma_x\;dt\\
\\
   \displaystyle-\int_{\Gamma_T}A\nabla z_2   n_2q_k\dfrac{\partial z_2}{\partial x_k}\; d\sigma_x\;dt=
0. 
\end{array}
\end{equation}
Let us note that this last identity is similar to \eqref{fid3'} except for the integral terms defined on the boundary $\Sigma$. This is due to the fact that $z_2$ is defined on $\Omega_2$ whose boundary is only $\Gamma$. 

\noindent By summing up the identities \eqref{fid3'} and \eqref{fid4'}, we obtain
\begin{equation}\label{fid5'}
\begin{array}{c}
\displaystyle\sum_{i=1}^2 \left.\left(z'_i,q_k\dfrac{\partial z_i}{\partial x_k}\right)_{\Omega_i}\right|_0^T+\dfrac{1}{2}\sum_{i=1}^2\int_{Q_i}\left(|z'_i|^2-A\nabla z_i\nabla z_i\right)\dfrac{\partial q_k}{\partial x_k}\;dx\;dt+ \sum_{i=1}^2\int_{Q_i} A\nabla z_i \nabla q_k\dfrac{\partial z_i}{\partial x_k}\;dx\;dt\\
\\
\displaystyle -\sum_{i=1}^2\dfrac{1}{2}\int_{Q_i}q_k\sum_{l,j=1}^n \dfrac{\partial a_{l,j}}{\partial x_k} \dfrac{\partial z_i}{\partial x_l} \dfrac{\partial z_i}{\partial x_j}\; dx\;dt
-\dfrac{1}{2}\sum_{i=1}^2\int_{\Gamma_T} \left|z'_i\right|^2q_k n_{ik}\;d\sigma_ x\;dt\\
\\
+\displaystyle \dfrac{1}{2}\sum_{i=1}^2\int_{\Gamma_T} A\nabla z_i\nabla z_i q_kn_{ik}\;d\sigma_x\;dt-\sum_{i=1}^2\int_{\Gamma_T}A\nabla z_i   n_iq_k\dfrac{\partial z_i}{\partial x_k}\; d\sigma_x\;dt\\
\\
\displaystyle-\dfrac{1}{2}\int_{\Sigma}A n_1n_1q_k n_{1k}\left(\dfrac{\partial z_1}{\partial n_1}\right)^2\; d\sigma_x\;dt=0. 
\end{array}
\end{equation}
The above identity is nearly close to the claimed one except for the third line. Nevertheless, if we observe that, for any fixed $i=1,2$ we have 
\begin{equation}\label{relgrad}
\nabla z_i=\dfrac{\partial z_i}{\partial n_i}n_i+\nabla_\sigma z_i
\end{equation}
on the interface $\Gamma$, by the symmetry of $A$, the third line of \eqref{fid5'} becomes 

\begin{equation}\label{fid6}
\begin{array}{l}
\displaystyle \dfrac{1}{2}\int_{\Gamma_T} A\nabla z_i\nabla z_i q_kn_{ik}\;d\sigma_x\;dt-\int_{\Gamma_T}A\nabla z_i   n_iq_k\dfrac{\partial z_i}{\partial x_k}\; d\sigma_x\;dt =\\
 \\
 
\displaystyle=\dfrac{1}{2}\int_{\Gamma_T} A\left(\dfrac{\partial z_i}{\partial n_i}n_i+\nabla_\sigma z_i\right)\left(\dfrac{\partial z_i}{\partial n_i}n_i+\nabla_\sigma z_i\right)q_kn_{ik}\;d\sigma_x\;dt\\
\\
\qquad\qquad\displaystyle-\int_{\Gamma_T}A\left(\dfrac{\partial z_i}{\partial n_i}n_i+\nabla_\sigma z_i\right)   n_iq_k\left(\dfrac{\partial z_i}{\partial n_i}n_{ik}+(\nabla_\sigma z_i)_k\right)\; d\sigma_x\;dt\\
\\
\displaystyle=\dfrac{1}{2}\int_{\Gamma_T} An_in_i\left(\dfrac{\partial z_i}{\partial n_i}\right)^2q_kn_{ik}\;d\sigma_x\;dt+\int_{\Gamma_T} An_i\nabla_\sigma z_i\dfrac{\partial z_i}{\partial n_i}q_kn_{ik}\;d\sigma_x\;dt+\dfrac{1}{2}\int_{\Gamma_T} A\nabla_\sigma z_i\nabla_\sigma z_i q_kn_{ik}\;d\sigma_x\;dt\\
\\
\qquad\displaystyle-\int_{\Gamma_T} An_in_i\left(\dfrac{\partial z_i}{\partial n_i}\right)^2q_kn_{ik}\;d\sigma_x\;dt-\int_{\Gamma_T}A\nabla_\sigma z_i n_i q_k\dfrac{\partial z_i}{\partial n_i}n_{ik}\;d\sigma_x\;dt-\int_{\Gamma_T}A\nabla z_i   n_iq_k(\nabla_\sigma z_i)_k\; d\sigma_x\;dt\\
\\
=\displaystyle-\dfrac{1}{2}\int_{\Gamma_T} An_in_i\left(\dfrac{\partial z_i}{\partial n_i}\right)^2q_kn_{ik}\;d\sigma_x\;dt+\dfrac{1}{2}\int_{\Gamma_T} A\nabla_\sigma z_i\nabla_\sigma z_i q_kn_{ik}\;d\sigma_x\;dt-\int_{\Gamma_T}A\nabla z_i   n_iq_k(\nabla_\sigma z_i)_k\; d\sigma_x\;dt.
\end{array}
\end{equation}
By putting \eqref{fid6} into \eqref{fid5'}, taking into account the interface condition in problem \eqref{eq2.6h} and since $n_2=-n_1$, we finally obtain the required identity
\begin{equation*}\label{fid7}
\begin{array}{l}
\displaystyle\sum_{i=1}^2 \left.\left(z'_i,q_k\dfrac{\partial z_i}{\partial x_k}\right)_{\Omega_i}\right|_0^T+\dfrac{1}{2}\sum_{i=1}^2\int_{Q_i}\left(|z'_i|^2-A\nabla z_i\nabla z_i\right)\dfrac{\partial q_k}{\partial x_k}\;dx\;dt+ \sum_{i=1}^2\int_{Q_i} A\nabla z_i \nabla q_k\dfrac{\partial z_i}{\partial x_k}\;dx\;dt\\
\\
\qquad\qquad\displaystyle -\sum_{i=1}^2\dfrac{1}{2}\int_{Q_i}q_k\sum_{l,j=1}^n \dfrac{\partial a_{l,j}}{\partial x_k} \dfrac{\partial z_i}{\partial x_l} \dfrac{\partial z_i}{\partial x_j}\; dx\;dt
-\dfrac{1}{2}\sum_{i=1}^2\int_{\Gamma_T} \left|z'_i\right|^2q_k n_{ik}\;d\sigma_ x\;dt+\\
\\
\qquad\qquad\displaystyle-\sum_{i=1}^2\dfrac{1}{2}\int_{\Gamma_T} An_in_i\left(\dfrac{\partial z_i}{\partial n_i}\right)^2q_kn_{ik}\;d\sigma_x\;dt+\sum_{i=1}^2\dfrac{1}{2}\int_{\Gamma_T} A\nabla_\sigma z_i\nabla_\sigma z_i q_kn_{ik}\;d\sigma_x\;dt+\\
\\
\qquad\qquad+\displaystyle\int_{\Gamma_T}h\left(z_1-z_2\right)q_k\left(\nabla_\sigma (z_1-z_2)\right)_k\; d\sigma_x\;dt-\int_{\Sigma}A n_1n_1q_k n_{1k}\left(\dfrac{\partial z_1}{\partial n_1}\right)^2\; d\sigma_x\;dt=0.
\end{array}
\end{equation*}
This completes the proof of the lemma. \end{proof}

\noindent At this point we want  to apply the above identity for a particular choice of the vector field $q$ in order to derive the observability estimate. To this aim we adapt to our context some arguments introduced in \cite{ki27*} and \cite{ki27**}. 
\\
Let $x^0\in \mathbb{R}^n$ and set
\begin{equation}\label{m}
m(x)=x-x^0=(x_k-x^0_k)_{k=1}^n.
\end{equation}
\noindent We divide the boundary $\partial\Omega$ into two parts, i.e.  
$$\partial\Omega(x^0)=\{x\in \partial\Omega:\,m(x)  n_1(x)=m_k(x)n_{1k}(x)>0\} \hbox{ and  }
\partial\Omega_\ast(x^0)=\partial\Omega\setminus\partial\Omega(x^0)$$ and we denote 
$$\Sigma(x^0)=\partial\Omega(x^0)\times (0,T) \hbox{ and  }\Sigma_\ast(x^0)=\partial\Omega_\ast(x^0)\times (0,T).$$
Further, let us define 
\begin{equation}\label{defRi}
R_i(x^0)=\max_{x\in\overline{\Omega}_i} |m(x)|\text{ for }
i=1,2
\hbox{ and  }
R(x^0)=\max_{x\in\overline{\Omega}} |m(x)|.
\end{equation}
Some remarks are in order. Usually, in the context of controllability problems, the point $x^0$ can be viewed as an observer and $\partial\Omega(x^0)$ is strictly related to the action region, where the control is acting. The choice of  $x^0$ gives various control regions according to the position of the observer and has advantages and disadvantages.
For example, if $\Omega$ is a circle, geometrically, $\partial\Omega(x^0)$ is concave to the observer. More in particular, if $x^0$ is a point inside $\Omega$, then $\partial\Omega(x^0) = \partial\Omega$, since the entire boundary is concave to any point inside. On the other hand, if $x^0$ is outside $\Omega$, then drawing the tangents from $x^0$, the boundary is divided into two parts, where $\partial\Omega(x^0)$ is concave to $x^0$ (related to the control region)  and $\partial\Omega_\ast (x^0)$ is convex to $x^0$ (not related to the control region). 
When dealing with internal controllability, the control region is a neighbourhood of  $\partial\Omega(x^0)$. In our case, due to the geometry of the domain, we need to introduce a further control set which is a neighbourhood of the whole interface.
As we will see later on, the choice of $x^0$ will play a fundamental role also on the control time (see Lemmas \ref{invineq4} and \ref{invineq5**}). 
In the following, we introduce the energy $E(t)$ of problem \eqref{eq2.6h}-\eqref{regip}
\begin{equation}\label{energyz}
\begin{array}{ll}
E(t) =&\displaystyle\dfrac{1}{2}\left[ \int_{\ouno}|z'_{1}(t)|^{2}dx+\int_{\odue}|z'_{2}(t)|^{2}dx+\int_{\ouno}A\nabla z_{1}(t)\,\nabla z_{1}(t)dx\right.\\
\\
&\qquad \qquad\left.\displaystyle+\int_{\odue}A\nabla z_{2}(t)\,\nabla z_{2}(t)dx+\int_{\Gamma}h\left|z_{1}(t)-z_{2}(t)\right|^{2}d\sigma_{x}\right].
\end{array}
\end{equation}
Let us note that $E(t)$ is conserved (see \cite{DFMC}, Lemma 4.1), that is 
\begin{equation}\label{cons}
E(t) = E(0), \quad \hbox{for all}\; t \in [0,T].
\end{equation}
We set
\begin{equation}\label{inv1}
\begin{array}{l}
S= \displaystyle\dfrac{1}{2}\int_\Sigma A  n_1  n_1\left(\dfrac{\partial z_1}{\partial n_1}\right)^2 m_k n_{1k}\;d\sigma_x\;dt+\dfrac{1}{2}\sum_{i=1}^{2}\int_{\Gamma_T} A  n_i  n_i\left(\dfrac{\partial z_i}{\partial n_i}\right)^2 m_k n_{ik}\;d\sigma_x\;dt+\\
\\
\qquad-\displaystyle\int_{\Gamma_T} h  \left(z_1-z_2\right)m_k\left(\nabla_\sigma (z_1-z_2)\right)_k\;d\sigma_x\;dt+ \dfrac{1}{2}\sum_{i=1}^{2}\int_{\Gamma_T} \left(|z'_i|^2-A  \nabla_\sigma z_i\nabla_\sigma z_i\right)m_k n_{ik}\;d\sigma_x\;dt.
\end{array}
\end{equation}

 We want to find a lower bound for $S$. To this aim we introduce a technical geometrical assumption concerning not only the position of the observer $x^0$ but also the geometry of the domain $\Omega_2$. This geometrical property will characterize the choice of the control region related to the interface (see Definition \ref{omega2} and Lemma \ref{invineq3}).

\begin{lemma}\label{invineq1}
Let us suppose that $\Omega_2$ is star-shaped with respect to a point $x^0\in \Omega_2$. Let $z=(z_1,z_2)$ the solution of problem \eqref{eq2.6h}-\eqref{regip}. Then, for any $T>0$, it holds
\begin{equation}\label{invineq1*}
\begin{array}{c}

S \geq \left[T\left(1-\dfrac{n R(x^0)M}{\alpha}\right)-2\max \left(\dfrac{ R(x^0)}{\sqrt{\alpha}}, \dfrac{(n-1)\sqrt{\alpha}}{2 h_0}\right)\right]E(0), 
\end{array}
\end{equation}
with $M$ defined as in \eqref{defm} and $S$ as in \eqref{inv1}.
\end{lemma}

\begin{proof}

\noindent We take $q_k=m_k=x_k-x^0_k$, for $k=1,\ldots ,n$, in the identity \eqref{firstid}. Then, $\nabla q_k = \nabla m_k = e_k$, where $e_k$ is the canonical  basis element. In particular $\displaystyle \frac{\partial m_k}{\partial x_k} = 1$ and thus 
 $\displaystyle \sum_{k=1}^{n}\frac{\partial m_k}{\partial x_k} = n.$ Hence, we have
 
\begin{equation}\label{inv1''}
\begin{array}{l}
S=\displaystyle\sum_{i=1}^2\left.\left(z'_i,m_k\dfrac{\partial z_i}{\partial x_k}\right)_{\Omega_i}\right |_{0}^{T} +\dfrac{n}{2}\sum_{i=1}^{2}\int_{Q_i} \left(|z'_i|^2-A  \nabla z_i\nabla z_i\right)\;dx\;dt+\\
\\
\qquad\qquad\qquad+\displaystyle\sum_{i=1}^{2}\int_{Q_i} A  \nabla z_i\nabla z_i\;dx\;dt-\dfrac{1}{2}\sum_{i=1}^{2}\int_{Q_i} m_k
\sum_{l,j=1}^{n}\dfrac{\partial a_{lj}}{\partial x_k}\dfrac{\partial z_i}{\partial x_l}\dfrac{\partial z_i}{\partial x_j}\;dx\;dt.\\
\\
\qquad = S_1 + S_2+S_3+S_4.
\end{array}
\end{equation}
We want to estimate $S_1 + S_2+S_3+S_4$. 
Let us pose
\begin{equation}\label{defXi}
X_i=\left.\left(z'_i(t),m_k\dfrac{\partial z_i(t)}{\partial x_k}\right)_{\Omega_i}\right |_{0}^{T}
\end{equation}
 and
\begin{equation}\label{defYi}
Y_i=\int_{Q_i} \left(|z'_i|^2-A  \nabla z_i\nabla z_i\right)\,dx\,dt
\end{equation}
for $i=1,2$. Hence, $S_1 = X_1 + X_2$, $S_2 =\dfrac{n}{2}(Y_1 +Y_2)$ and therefore \eqref{inv1''} can be rewritten as
%
\begin{equation}\label{inv3}
\begin{array}{c}
 S= S_1 + S_2+S_3+S_4 =\displaystyle \left(X_1+X_2\right) +\dfrac{n-1}{2}\left(Y_1+Y_2\right)+\dfrac{1}{2}\sum_{i=1}^{2}\int_{Q_i} \left(|z'_i|^2+A  \nabla z_i\nabla z_i\right)\;dx\;dt\\
\\
\displaystyle-\dfrac{1}{2}\sum_{i=1}^{2}\int_{Q_i} m_k
\sum_{l,j=1}^{n}\dfrac{\partial a_{lj}}{\partial x_k}\dfrac{\partial z_i}{\partial x_l}\dfrac{\partial z_i}{\partial x_j}\;dx\;dt.
\end{array}
\end{equation}
Taking into account \eqref{energyz} and the conservation law \eqref{cons}, we get
\begin{equation}\label{inv4'}
\begin{array}{l}
 S_1 + S_2+S_3+S_4 =\displaystyle \left(X_1+X_2\right) +\dfrac{n-1}{2}\left(Y_1+Y_2\right)+E(0)T\\
\\
\qquad\qquad\qquad\qquad\displaystyle -\dfrac{1}{2}\int_{\Gamma_T}h  \left|z_{1}-z_{2}\right|^{2}d\sigma_{x}\,dt
\displaystyle-\dfrac{1}{2}\sum_{i=1}^{2}\int_{Q_i} m_k
\sum_{l,j=1}^{n}\dfrac{\partial a_{lj}}{\partial x_k}\dfrac{\partial z_i}{\partial x_l}\dfrac{\partial z_i}{\partial x_j}\;dx\;dt.
\end{array}
\end{equation}
By multiplying the PDEs in  \eqref{eq2.6h} by $z_1$ and $z_2$ respectively, and taking into account interface and boundary conditions, we get
\begin{equation*}
Y_1+Y_2=\sum_{i=1}^{2}\left.\left(z'_i(t),z_i(t)\right)_{\Omega_i}\right|_{0}^{T}+\int_{\Gamma_T}h  \left(z_1-z_2\right)^2\,d\sigma_x\,dt.
\end{equation*}
Hence \eqref{inv4'} becomes
\begin{equation}\label{inv4}
\begin{array}{c}
S_1 + S_2+S_3+S_4=\displaystyle Z_1+Z_2+E(0)T+\dfrac{n-2}{2}\int_{\Gamma_T}h  \left|z_{1}(t)-z_{2}(t)\right|^{2}d\sigma_{x}\,dt
\\
\\
\displaystyle\qquad\qquad\qquad-\dfrac{1}{2}\sum_{i=1}^{2}\int_{Q_i} m_k
\sum_{l,j=1}^{n}\dfrac{\partial a_{lj}}{\partial x_k}\dfrac{\partial z_i}{\partial x_l}\dfrac{\partial z_i}{\partial x_j}\;dx\;dt,
\end{array}
\end{equation}
where 
\begin{equation}\label{defZi}
Z_i=\left.\left(z'_i(t),m_k\dfrac{\partial z_i(t)}{\partial x_k}+\dfrac{n-1}{2}z_i(t)\right)_{\Omega_i}\right|_{0}^{T}, \quad \text{ for }i=1,2.
\end{equation}
Thus, we have an $E(0)T$ term. We need to see that it is a leading term. Thus, we need to  estimate the other terms in \eqref{inv4}. To this aim, let us fix $i\in\{1,2\}$. By Young inequality we get
\begin{equation}\label{inv5}
\begin{array}{ll}
\displaystyle \left|\left(z'_i(t),m_k\dfrac{\partial z_i(t)}{\partial x_k}+\dfrac{n-1}{2}z_i(t)\right)_{\Omega_i}\right|&\displaystyle\leq\int_{\Omega_i}\left|z_i'(t)\right|\left|m_k\dfrac{\partial z_i(t)}{\partial x_k}+\dfrac{n-1}{2}z_i(t)\right|\,dx\\
\\
&\displaystyle \leq \dfrac{\mu}{2}\int_{\Omega_i}\left|z_i'(t)\right|^2\,dx+\dfrac{1}{2\mu}\int_{\Omega_i}\left| m_k\dfrac{\partial z_i(t)}{\partial x_k}+\dfrac{n-1}{2}z_i(t)\right|^2\,dx.
\end{array}
\end{equation}
By applying Gauss-Green, it holds
\begin{equation*}
\int_{\Omega_i}m_k\dfrac{\partial z_i(t)}{\partial x_k}z_i(t)\,dx=\dfrac{1}{2}\int_{\Omega_i}m_k\dfrac{\partial}{\partial x_k}\left|z_i(t)\right|^2=-\dfrac{n}{2}\int_{\Omega_i}\left|z_i(t)\right|^2\,dx+\dfrac{1}{2}\int_{\Gamma}m_kn_{ik}\left|z_i(t)\right|^2\,d\sigma_x.
\end{equation*}
Hence, by \eqref{defRi}, the second term in the right hand side of \eqref{inv5} can be estimated as 
\begin{equation}\label{inv6}
\begin{array}{ll}
\displaystyle\int_{\Omega_i}\left| m_k\dfrac{\partial z_i(t)}{\partial x_k}+\dfrac{n-1}{2}z_i(t)\right|^2\,dx &=\displaystyle\int_{\Omega_i}\left| m_k\dfrac{\partial z_i(t)}{\partial x_k}\right|^2\,dx+\left[\dfrac{(n-1)^2}{4}-\dfrac{n(n-1)}{2}\right]\int_{\Omega_i}\left|z_i(t)\right|^2\,dx\\
\\
&\qquad\qquad\displaystyle +\dfrac{n-1}{2}\int_{\Gamma}m_kn_{ik}\left|z_i(t)\right|^2\,d\sigma_x\\
\\
&\displaystyle\leq (R_i(x^0))^2\int_{\Omega_i}\left|\nabla z_i(t)\right|^2\,dx+\dfrac{n-1}{2}\int_{\Gamma}m_kn_{ik}\left|z_i(t)\right|^2\,d\sigma_x,
\end{array}
\end{equation}
where, we have used the fact that $\dfrac{(n-1)^2}{4}-\dfrac{n(n-1)}{2} <0$. Let us note that by \eqref{defRi} 
\begin{equation}\label{RRi}
R(x^0)=R_1(x^0)>R_2(x^0),
\end{equation}
 since $x^0\in \Omega_2$,
thus, by putting \eqref{inv6} into \eqref{inv5} and taking into account \eqref{ipoA}, we obtain
\begin{equation}\label{inv7}
\begin{array}{ll}
\displaystyle \left|\left(z'_i(t),m_k\dfrac{\partial z_i(t)}{\partial x_k}+\dfrac{n-1}{2}z_i(t)\right)_{\Omega_i}\right|&\displaystyle\leq \dfrac{\mu}{2}\int_{\Omega_i}\left|z_i'(t)\right|^2\,dx+\dfrac{(R(x^0))^2}{2\alpha\mu}\int_{\Omega_i}A  \nabla z_i(t)\nabla z_i(t)\,dx\\
\\
&\displaystyle \qquad\qquad+\dfrac{n-1}{4\mu}\int_{\Gamma}m_kn_{ik}\left|z_i(t)\right|^2\,d\sigma_x.
\end{array}
\end{equation}
Let us consider the last term in \eqref{inv7}, for $i=1,2$. We observe that
$$
\left|\dfrac{1}{2}|z_1(t)|^2-|z_2(t)|^2\right|\leq \left|z_1(t)-z_2(t)\right|^2,\quad \forall \; t\in[0,T].
$$
Moreover, by our assumption on $x^0$ and since $n_1=-n_2$ on $\Gamma$, it holds that $m_k n_{1k}\leq 0$ on $\Gamma$. Hence,  we get
\begin{equation}\label{inv9}
\begin{array}{ll}
\displaystyle \sum_{i=1}^2\int_{\Gamma}m_kn_{ik}\left|z_i(t)\right|^2\,d\sigma_x&\displaystyle =\int_{\Gamma}m_kn_{1k}\left(\left|z_1(t)\right|^2-\left|z_2(t)\right|^2\right)\,d\sigma_x\\
\\
&\leq \displaystyle\int_{\Gamma}m_kn_{1k}\left(\dfrac{1}{2}\left|z_1(t)\right|^2-\left|z_2(t)\right|^2\right)\,d\sigma_x\\
\\
&\displaystyle \leq \int_{\Gamma}|m_kn_{1k}|\left|\dfrac{1}{2}\left|z_1(t)\right|^2-\left|z_2(t)\right|^2\right|\,d\sigma_x \\
\\
&\leq  \displaystyle \|m\|_{L^\infty(\Gamma)}\int_{\Gamma}\left|\dfrac{1}{2}\left|z_1(t)\right|^2-\left|z_2(t)\right|^2\right|\,d\sigma_x\\
\\
&\displaystyle\leq \dfrac{R(x^0)}{h_0}\int_{\Gamma}h  \left(z_1(t)-z_2(t)\right)^2\,d\sigma_x.
\end{array}
\end{equation}
Taking into account \eqref{inv7} and \eqref{inv9}, we get the estimate
\begin{equation}\label{inv10}
\begin{array}{ll}
\displaystyle \sum_{i=1}^2\left|\left(z'_i(t),m_k\dfrac{\partial z_i(t)}{\partial x_k}+\dfrac{n-1}{2}z_i(t)\right)_{\Omega_i}\right|&\leq
\displaystyle  \dfrac{\mu}{2}\sum_{i=1}^2\int_{\Omega_i}\left|z_i'(t)\right|^2\,dx\\
\\&\qquad\displaystyle+\dfrac{(R(x^0))^2}{2\alpha\mu}\sum_{i=1}^2 \int_{\Omega_i}A  \nabla z_i(t)\nabla z_i(t)\,dx\\
\\
&\qquad\displaystyle +\dfrac{(n-1)R(x^0)}{4\mu h_0}\int_{\Gamma}h  \left(z_1(t)-z_2(t)\right)^2\,d\sigma_x\\
\\ \displaystyle &\leq \max \left(\dfrac{R(x^0)}{\sqrt{\alpha}},\dfrac{(n-1)\sqrt{\alpha}}{2 h_0}\right)E(t).
\end{array}
\end{equation}
The last inequality follows by choosing $\mu=\dfrac{R(x^0)}{\sqrt{\alpha}}$ and by the definition of energy as in  \eqref{energyz}.
%
Hence by \eqref{cons} and taking into account \eqref{defZi}, we readily see that
\begin{equation}\label{inv11}
	\displaystyle\left|Z_1+Z_2\right|\leq 
	2\max \left(\dfrac{R(x^0)}{\sqrt{\alpha}},\dfrac{(n-1)\sqrt{\alpha}}{2 h_0}\right)E(0).
\end{equation}

The estimate of the last term in \eqref{inv4} is straight forward using the ellipticity and boundedness of the matrix A (see also \cite{LW}). More precisely, taking into account the energy  definition in \eqref{energyz}, we have
\begin{equation}\label{inv12}
\begin{array}{ll}
\displaystyle\left|\dfrac{1}{2}\sum_{i=1}^{2}\int_{Q_i} m_k
\sum_{l,j=1}^{n}\dfrac{\partial a_{lj}}{\partial x_k}\dfrac{\partial z_i}{\partial x_l}\dfrac{\partial z_i}{\partial x_j}\;dx\;dt\right| &\leq \displaystyle\sum_{i=1}^{n}\dfrac{n R_i(x^0)M}{2\alpha}\int_{Q_i}A\nabla z_i\nabla z_i\,dx\,dt
\\
\\
&\leq \dfrac{n R(x^0)M}{\alpha}TE(0).
\end{array}
\end{equation}
By putting \eqref{inv11} and \eqref{inv12} into \eqref{inv4}, we finally arrive at the lower bound
\begin{equation}\label{inv13}
\begin{array}{c}
S\geq \displaystyle -2\max \left(\dfrac{R(x^0)}{\sqrt{\alpha}}, \dfrac{(n-1)\sqrt{\alpha}}{2 h_0}\right)E(0)+E(0)T-\dfrac{n R(x^0)M}{\alpha}TE(0).
\end{array}
\end{equation}
The proof is now complete.
\end{proof}

\medskip

\noindent We now specify the required topological assumptions on the control regions $\omega_1$ and $\omega_2$ in order to obtain our exact controllability result. See Fig. 1, Fig. 2 and Fig. 3 for sample domains.

\begin{definition}\label{omega1}
Let $x^0$ be as in the hypotheses of Lemma \ref{invineq1}. We say that $\omega_1\subset \Omega_1$ is a neighbourhood of ${\partial\Omega(x^0)}$ if there exists some neighbourhood $\mathcal{O}\subset\mathbb{R}^n$ of $\partial{\Omega(x^0)}$ such that
$$
\omega_1=\Omega_1\cap
\mathcal{O}.
$$
\end{definition}
\begin{definition}\label{omega2}
We say that $\omega_2\subset \Omega_2$ is a neighborhood of ${\Gamma}$, if there exists some neighborhood $\mathcal{O}\subset\mathbb{R}^n$ of ${\Gamma}$ such that
$$
\omega_2=\Omega_2\cap
\mathcal{O}.
$$
\end{definition}
\noindent We will now establish a couple of important results  which are crucial to get the observability inequality given in Lemma \ref{invineq5**}  below. In this direction, we consider the function $\tau=(\tau_1, \ldots \tau_k) \in (C^1(\mathbb{R}^n))^n$ satisfying the following properties:
\begin{equation}\label{conh}
\left\{
\begin{array}{lll}
i) & \tau \cdot n_1 =1\; \text{on } \partial \Omega,\\
\\
ii) &\text{supp }\tau \subset\omega_1,\\
\\
iii) &\|\tau\|_{(L^\infty(\omega_1))^n}\leq 1.
\end{array}
\right.
\end{equation}
The existence of such a vectorial field is proved in \cite{ki27*}. 
\begin{lemma}\label{invineq2}
Let $\omega_1$ be a neighborhood of ${\partial\Omega(x^0)}$ and let $z=(z_1,z_2)$ the solution of problem \eqref{eq2.6h}-\eqref{regip}. Then, for any $T>0$, it holds
\begin{equation}\label{invineq2*}
\dfrac{1}{2}\left|\int_{\Sigma(x^0)} A  n_1\,n_1\left(\dfrac{\partial z_1}{\partial n_1}\right)^2\;d\sigma_x\;dt\right| \leq  2 \max \left(1, \dfrac{1}{\alpha}\right) E(0) + C \int_0^T \int_{\omega_1}\left(\left|z_1'\right|^2+\left|\nabla z_1\right|^2\right)\,dx\,dt.
\end{equation}
\end{lemma}
\begin{proof} Taking $q_k=\tau_k$, $k=1, \ldots, n$,  in \eqref{firstid}, by \eqref{conh}i) and   \eqref{conh}ii) we get
\begin{equation}\label{firstidh}
\begin{array}{ll}
\displaystyle\dfrac{1}{2}\int_{\Sigma(x^0)} A  n_1  n_1\left(\dfrac{\partial z_1}{\partial n_1}\right)^2\;d\sigma_x\;dt&=\displaystyle\left.\left(z'_1,\tau_k\dfrac{\partial z_1}{\partial x_k}\right)_{\omega_1}\right |_{0}^{T} +\dfrac{1}{2}\int_0^T \int_{\omega_1} \left(|z'_1|^2-A  \nabla z_1\nabla z_1\right)\dfrac{\partial \tau_k}{\partial x_k}\;dx\;dt+\\
\\
&+\displaystyle\int_0^T \int_{\omega_1} A  \nabla z_1\nabla \tau_k\dfrac{\partial z_1}{\partial x_k}\;dx\;dt-\dfrac{1}{2}\int_0^T \int_{\omega_1} \tau_k
\sum_{l,j=1}^{n}\dfrac{\partial a_{lj}}{\partial x_k}\dfrac{\partial z_1}{\partial x_l}\dfrac{\partial z_1}{\partial x_j}\;dx\;dt.
\end{array}
\end{equation}
Passing to  the absolute value, by \eqref{energyz}, \eqref{conh}iii), Young inequality, the conservation law and since $\tau \in (C^1(\mathbb{R}^n))^n$ , we obtain
\begin{equation*}
\begin{array}{ll}
\displaystyle\dfrac{1}{2}\left|\int_{\Sigma(x^0)} A  n_1  n_1\left(\dfrac{\partial z_1}{\partial n_1}\right)^2\;d\sigma_x\;dt\right|\leq &=\displaystyle\dfrac{1}{2}\int_{\omega_1}\left|z_1'(0)\right|^2\,dx+\dfrac{1}{2}\int_{\omega_1}\left|z_1'(T)\right|^2\,dx+\dfrac{1}{2}\int_{\omega_1}\left|\nabla z_1(0)\right|^2\,dx\\
\\
&\qquad\displaystyle+\dfrac{1}{2}\int_{\omega_1}\left|\nabla z_1(T)\right|^2\,dx+C_1 \int_0^T \int_{\omega_1} \left(|z'_1|^2+A  \nabla z_1\nabla z_1\right)\,dx\,dt\\
\\
&\qquad+\displaystyle C_2 \int_0^T \int_{\omega_1} A  \nabla z_1\nabla z_1\;dx\;dt+C_3 \int_0^T \int_{\omega_1}\left|\nabla z_1\right|^2\;dx\;dt\leq\\
\\
&\leq \displaystyle   2 \max \left(1, \dfrac{1}{\alpha}\right) E(0)+C \int_0^T \int_{\omega_1} \left(|z'_1|^2+\left|\nabla z_1\right|^2\right) \;dx\;dt.
\end{array}
\end{equation*}
\end{proof}
We will get a similar result for the neighborhood $\omega_2$. To this aim, let $w\in C^1(\mathbb{R}^n)$ be such that
\begin{equation}\label{conh*}
\left\{
\begin{array}{lll}
i)&\text{supp }w\subset \omega_2, & 
\\ \\
ii)& 0\leq w \leq 1\; \text{in }\; \omega_2,\\
\\
iii)& w=1 \; \text{on }\; \Gamma, \\
\\
iv)& \|\nabla w\|_{L^\infty(\mathbb{R}^n)}\leq C.
\end{array}
\right.
\end{equation}
The existence of $w$ is quite standard, see for example, \cite{ki27*}. 
\\
Let us denote 
\begin{equation*}
\begin{array}{lll}
S_{\Gamma_T}&=\displaystyle\dfrac{1}{2}\sum_{i=1}^{2}\int_{\Gamma_T} A  n_i  n_i\left(\dfrac{\partial z_i}{\partial n_i}\right)^2 m_k n_{ik}\;d\sigma_x\;dt -\int_{\Gamma_T} h  \left(z_1-z_2\right)m_k\left(\nabla_\sigma (z_1-z_2)\right)_k\;d\sigma_x\;dt\\
\\
&\qquad\qquad \displaystyle +\dfrac{1}{2}\sum_{i=1}^{2}\int_{\Gamma_T} \left(|z'_i|^2-A  \nabla_\sigma z_i\nabla_\sigma z_i\right) m_k n_{ik}\;d\sigma_x\;dt. 
\end{array}
\end{equation*} 
we want to find an upper bound for $S_{\Gamma_T}$.
\begin{lemma}\label{invineq3}
Let $\omega_2$ be a neighborhood of ${\Gamma}$ and let $z=(z_1,z_2)$ the solution of problem \eqref{eq2.6h}-\eqref{regip}. Then, for any $T>0$, it holds
\begin{equation}\label{invineq3*}
\begin{array}{c}
\left|S_{\Gamma_T}\right| \displaystyle\leq  2\max\left(1,\dfrac{R^2(x^0)}{\alpha}\right)E(0) + C\int_0^T \int_{\omega_2}\left(\left|z_2'\right|^2+\left|\nabla z_2\right|^2\right)\,dx\,dt.
\end{array}
\end{equation}
\end{lemma}

 \begin{proof} Let us choose in \eqref{firstid} $q_k=m_k\,w$, $k=1, \ldots, n$. By Young inequality, \eqref{energyz} and the conservation law, we obtain
 \begin{equation*}
\begin{array}{lll}
\left|S_{\Gamma_T}\right|&\displaystyle\leq\dfrac{1}{2}\int_{\omega_{2}}\left|z_2'(0)\right|^2\,dx+\dfrac{1}{2}\int_{\omega_2}\left|z_2'(T)\right|^2\,dx+\dfrac{R^2(x^0)}{2}\int_{\omega_2}\left|\nabla z_2(0)\right|^2\,dx\\
\\
&\qquad\qquad\displaystyle+\dfrac{R^2(x^0)}{2}\int_{\omega_{2}}\left|\nabla z_2(T)\right|^2\,dx+\dfrac{n}{2}\int_0^T \int_{\omega_2} \left(|z'_2|^2+A  \nabla z_2\nabla z_2\right)\;dx\;dt\\
\\ 
&\qquad\qquad\displaystyle +
C_0 \int_0^T \int_{\omega_2} \left(|z'_2|^2+A  \nabla z_2\nabla z_2\right)\;dx\;dt
 \\ \\
 &\qquad\qquad+\displaystyle C_1 \int_0^T \int_{\omega_2} A  \nabla z_2\nabla z_2\;dx\;dt+C_2 \int_0^T \int_{\omega_2}\left|\nabla z_2\right|^2\;dx\;dt\\
\\
&\leq \displaystyle 2\max\left(1,\dfrac{R^2(x^0)}{\alpha}\right)E(0)+C \int_0^T \int_{\omega_2} \left(|z'_2|^2+\left|\nabla z_2\right|^2\right) \;dx\;dt.
\end{array}
\end{equation*} 
 \end{proof}

\medskip
\noindent Collecting together the results of Lemma \ref{invineq1}, Lemma \ref{invineq2} and  Lemma \ref{invineq3}, we obtain the following lower estimate.
\begin{lemma}\label{invineq4}
Let us suppose that $\Omega_2$ is star-shaped with respect to a point $x^0\in \Omega_2$ satisfying
\begin{equation}\label{condr0}
R(x^0)<\frac{\alpha}{nM}.
\end{equation} Assume  $\omega_1$ and $\omega_2$ are neighbourhoods of ${\partial\Omega(x^0)}$ and $\Gamma$ respectively, and let $z=(z_1,z_2)$ the solution of problem \eqref{eq2.6h}-\eqref{regip}. Then, there exists $T_0>0$ such that
\begin{equation}\label{invineq4*}
 E(0)\leq  C_1(T)\int_0^T \int_{\omega_1}\left(\left|z_1'\right|^2+\left|\nabla z_1\right|^2\right)\,dx\,dt+ C_2(T)\int_0^T \int_{\omega_2}\left(\left|z_2'\right|^2+\left|\nabla z_2\right|^2\right)\,dx\,dt,
\end{equation}
for $T$ large enough so that
\begin{equation}\label{condT'}
\dfrac{T-T_0}{T}>\dfrac{n R(x^0)M}{ \alpha}.
\end{equation}

\end{lemma}
\begin{proof}
By putting \eqref{invineq2*} and \eqref{invineq3*} into \eqref{invineq1*}, we get 
\begin{equation*}
\begin{array}{c}
\left[T\left(1-\dfrac{n R(x^0)M}{\alpha}\right)-2\max \left(\dfrac{R(x^0)}{\sqrt{\alpha}} ,\dfrac{(n-1)\sqrt{\alpha}}{2 h_0}\right)\right]E(0)\\\\\leq 2 \max\left(1, R(x^0), \dfrac{R(x^0)}{\alpha},\dfrac{R^2(x^0)}{\alpha}\right)E(0)+\\
\\
 \displaystyle +C_1\int_0^T \int_{\omega_1}\left(\left|z_1'\right|^2+\left|\nabla z_1\right|^2\right)\,dx\,dt+ C_2\int_0^T \int_{\omega_2}\left(\left|z_2'\right|^2+\left|\nabla z_2\right|^2\right)\,dx\,dt.
\end{array}
\end{equation*}
Denoting 
\begin{equation}\label{t0}
T_0=2\max \left(\dfrac{R(x^0)}{\sqrt{\alpha}}, \dfrac{(n-1)\sqrt{\alpha}}{2 h_0}\right)+2 \max\left(1, R(x^0), \dfrac{R(x^0)}{\alpha},\dfrac{R^2(x^0)}{\alpha}\right), 
\end{equation}
we obtain
\begin{equation*}
\begin{array}{ll}
\left[T\left(1-\dfrac{n R(x^0)M}{\alpha}\right)-T_0\right]E(0)
&\leq \displaystyle C_1\int_0^T \int_{\omega_1}\left(\left|z_1'\right|^2+\left|\nabla z_1\right|^2\right)\,dx\,dt\\
\\
 &\qquad\qquad+\displaystyle C_2\int_0^T \int_{\omega_2}\left(\left|z_2'\right|^2+\left|\nabla z_2\right|^2\right)\,dx\,dt.
\end{array}
\end{equation*}
Thus,  if \eqref{condr0} is satisfied and if $T$ is large enough so that \eqref{condT'} holds, $T\left(1-\dfrac{n R(x^0)M}{\alpha}\right)-T_0$ is positive and we get the result.
\end{proof}
\noindent Some comments are in order.
For sake of simplicity, all integrals in previous lemmas are written between $0$ and $T$. Actually they could, as well, have been written between $\varepsilon$ and $T-\varepsilon$ with $\varepsilon>0$ and sufficiently small. 
More precisely, by using \eqref{energyz} and the conservation law, the inequalities  \eqref{invineq1*}, \eqref{invineq2*} and \eqref{invineq3*} can be written as

\begin{equation*}
\begin{array}{ll}
\displaystyle\dfrac{1}{2}\int_\varepsilon^{T-\varepsilon}\int_{\partial\Omega} A  n_1  n_1\left(\dfrac{\partial z_1}{\partial n_1}\right)^2 m_k n_{1k}\;d\sigma_x\;dt
+\displaystyle\dfrac{1}{2}\sum_{i=1}^{2}\int_\varepsilon^{T-\varepsilon}\int_{\Gamma} A  n_i  n_i\left(\dfrac{\partial z_i}{\partial n_i}\right)^2 m_k n_{ik}\;d\sigma_x\;dt\\
\\
-\displaystyle\int_\varepsilon^{T-\varepsilon}\int_{\Gamma} h  \left(z_1-z_2\right)m_k\left(\nabla_\sigma (z_1-z_2)\right)_k\;d\sigma_x\;dt
+\displaystyle \dfrac{1}{2}\sum_{i=1}^{2}\int_\varepsilon^{T-\varepsilon}\int_{\Gamma} \left(|z'_i|^2-A  \nabla_\sigma z_i\nabla_\sigma z_i\right)m_k n_{ik}\;d\sigma_x\;dt\\
\\
\qquad\qquad\qquad\geq \left[(T-2\varepsilon)\left(1-\dfrac{n R(x^0)M}{2\alpha}\right)-2\max \left(\dfrac{R(x^0)}{\sqrt{\alpha}}, \dfrac{(n-1)\sqrt{\alpha}}{2 h_0}\right)\right]E(0),
\end{array}
\end{equation*}

\begin{equation*}
\begin{array}{ll}
\displaystyle\dfrac{1}{2}\left|\int_\varepsilon^{T-\varepsilon}\int_{\partial \Omega(x^0)} A  n_1  n_1\left(\dfrac{\partial z_1}{\partial n_1}\right)^2 m_k n_{1k}\;d\sigma_x\;dt\right| &\leq 2 \max\left(1,\dfrac{1}{\alpha}\right) E(0) + \\
\\
&\qquad+\displaystyle C_1\int_\varepsilon^{T-\varepsilon}\int_{\omega_1}\left(\left|z_1'\right|^2+\left|\nabla z_1\right|^2\right)\,dx\,dt,
\end{array}\end{equation*}
and
\begin{equation*}
\begin{array}{c}
 \displaystyle\dfrac{1}{2}\left|\sum_{i=1}^{2}\int_\varepsilon^{T-\varepsilon}\int_{\Gamma} A  n_i  n_i\left(\dfrac{\partial z_i}{\partial n_i}\right)^2 m_k n_{ik}\;d\sigma_x\;dt-\int_\varepsilon^{T-\varepsilon}\int_{\Gamma} h  \left(z_1-z_2\right)m_k\left(\nabla_\sigma (z_1-z_2)\right)_k\;d\sigma_x\;dt+\right.	\\
\\
 \displaystyle+\left. \dfrac{1}{2}\sum_{i=1}^{2}\int_\varepsilon^{T-\varepsilon}\int_{\Gamma} \left(|z'_i|^2-A  \nabla_\sigma z_i\nabla_\sigma z_i\right)m_k n_{ik}\;d\sigma_x\;dt\right| \leq\\
 \\
 \displaystyle\leq  2\max\left(1,\dfrac{R^2(x^0)}{\alpha}\right)E(0) + C_2\int_\varepsilon^{T-\varepsilon}\int_{\omega_2}\left(\left|z_2'\right|^2+\left|\nabla z_2\right|^2\right)\,dx\,dt,
\end{array}
\end{equation*}
respectively. 

By arguing as in Lemma \ref{invineq4}, if $\varepsilon$ is chosen to have $T-2\varepsilon$ large enough so that
\begin{equation*}
\dfrac{(T-2\varepsilon)-T_0}{(T-2\varepsilon)}>\dfrac{n R(x^0)M}{ \alpha}
\end{equation*} 
and if \eqref{condr0} is satisfied, we get
\begin{equation}\label{invineq4**}
 E(0)\leq  C_1(T)\int_\varepsilon^{T-\varepsilon}\int_{\omega_1}\left(\left|z_1'\right|^2+\left|\nabla z_1\right|^2\right)\,dx\,dt+ C_2(T)\int_\varepsilon^{T-\varepsilon}\int_{\omega_2}\left(\left|z_2'\right|^2+\left|\nabla z_2\right|^2\right)\,dx\,dt.
\end{equation}

Now, we can  prove the observability inequality which is crucial to establish the exact controllability result. 
\begin{lemma}[observability inequality]\label{invineq5**}
Let us suppose that $\Omega_2$ is star-shaped with respect to a point $x^0\in \Omega_2$ satisfying condition \eqref{condr0}. Assume that $\omega_1$ and $\omega_2$ are neighbourhoods of ${\partial\Omega(x^0)}$ and $\Gamma$ respectively and let $z=(z_1,z_2)$ be the solution of problem \eqref{eq2.6h}-\eqref{regip}. Then there exists $T_0>0$ such that
\begin{equation}\label{invineq5}
 E(0)\leq  C_1(T)\int_0^T \int_{\omega_1}\left(\left|z_1'\right|^2+\left|z_1\right|^2\right)\,dx\,dt+ C_2(T)\int_0^T \int_{\omega_2}\left(\left|z_2'\right|^2+\left|z_2\right|^2\right)\,dx\,dt,
\end{equation}
for $T$ large enough so that
\begin{equation}\label{condT}
\dfrac{T-T_0}{T}>\dfrac{R(x^0)M}{ \alpha}.
\end{equation}
\end{lemma}
\begin{proof} In view of \eqref{invineq4*}, we need to estimate $\nabla z_i$ in terms of $z_i$ and $z_i'$,  $i=1,2$.
Let $\omega_{01}\subset\Omega_1$ be a neighborhood of $\partial\Omega(x^0)$ and $\omega_{02}\subset\Omega_2$ be a neighborhood of $\Gamma$ such that
$$
\Omega\cap\omega_{0i}\subset\omega_i,\, i=1,2.
$$

\noindent Note that \eqref{invineq4**} is true for any neighborhood of $\partial\Omega(x^0)$ and $\Gamma$, then it is also true for $\omega_{0i}$, $i=1,2$ and we obtain
\begin{equation}\label{invineq5*}
 E(0)\leq  C'_1(T)\int_\varepsilon^{T-\varepsilon}\int_{\omega_{01}}\left(\left|z_1'\right|^2+\left|\nabla z_1\right|^2\right)\,dx\,dt+ C'_2(T)\int_\varepsilon^{T-\varepsilon}\int_{\omega_{02}}\left(\left|z_2'\right|^2+\left|\nabla z_2\right|^2\right)\,dx\,dt.
\end{equation}
Let us consider $\rho\in W^{1,\infty}(\Omega)$, $\rho\geq0$ such that
\begin{equation*}\label{proprho}
\left\{
\begin{array}{lll}
i)& \rho(x)=1\; \text{in }(\omega_{01}\cup \omega_{02}),\\
\\
ii)&\rho(x)=0\; \text{in }\Omega\setminus (\omega_{1}\cup \omega_{2}).
\end{array}
\right.
\end{equation*}
Define the function $p(x,t)=\eta(t)\rho(x)$ in $\Omega\times (0,T)$, where  $\eta(t)\in C^1([0,T])$ is such that $\eta(0)=\eta(T)=0$ and $\eta(t)=1$ in $(\varepsilon, T-\varepsilon)$. Thus, $p$ satisfies
\begin{equation}\label{propp}
\left\{
\begin{array}{lll}
i)&p(x,t)=1\; \text{in }(\omega_{01}\cup \omega_{02})\times (\varepsilon, T-\varepsilon),\\
\\
ii)&p(x,t)=0\;\text{in }(\Omega\setminus (\omega_{1}\cup \omega_{2}))\times (\varepsilon, T-\varepsilon),\\
\\
iii)&p(x,0)=p(x,T)=0 \; \text{in }\Omega,\\
\\
iv)&\dfrac{|\nabla p|^2}{p}\in L^\infty(\Omega\times (0,T)).
\end{array}
\right.
\end{equation}
Multiplying the equation for $z_1$ in \eqref{eq2.6h} by $pz_1$ and integrating by parts in $Q_1$, we obtain
\begin{equation}\label{fid1*}
\begin{array}{c}
\displaystyle \int_0^T\langle z_{1}'', pz_1 \rangle_{V', V}\;dt +  \int_{\omega_1\times (0,T)}A \nabla z_{1}  \nabla \left(pz_1\right)\; dx\;dt
-\int_{\Gamma_T}A\nabla z_1pz_1  n_1\; d\sigma_x\;dt=0
\end{array}
\end{equation}
using \eqref{propp}$\rm{ii)}$ and the fact that $z_1 =0$ on $\Sigma$.
One more integration by parts of the first term leads to
\begin{equation*}
\begin{array}{c}
 \displaystyle- \int_{0}^T\int_{\omega_{1}}z'_1p'z_1\; dx\;dt- \int_{0}^T\int_{\omega_{1}}|z'_1|^2p\; dx\;dt+\int_{0}^T\int_{\omega_{1}}A \nabla z_{1}  \nabla pz_1\; dx\;dt+\\
 \\
 \displaystyle\int_{0}^T\int_{\omega_{1}}A \nabla z_{1} \nabla z_1 p\; dx\;dt-\int_{\Gamma_T}A\nabla z_1pz_1  n_1\; d\sigma_x\;dt=0.
\end{array}
\end{equation*}
Arguing as above, we get a similar identity for $z_2$.
Now, summing up and using the imperfect interface condition, we get
\begin{equation}\label{fid2}
\begin{array}{c}
  \displaystyle\int_{0}^T\int_{\omega_{1}}A \nabla z_{1} \nabla z_1 p\; dx\;dt+ \displaystyle\int_{0}^T\int_{\omega_{2}}A \nabla z_{2} \nabla z_2 p\; dx\;dt=\\
  \\
 \displaystyle\int_{0}^T\int_{\omega_{1}}|z'_1|^2p\; dx\;dt+ \int_{0}^T\int_{\omega_{2}}|z'_2|^2p\; dx\;dt+ \displaystyle \int_{0}^T\int_{\omega_{1}}z'_1p'z_1\; dx\;dt+\int_{0}^T\int_{\omega_{2}}z'_2p'z_2\; dx\;dt\\
  \\
 \displaystyle -\int_{0}^T\int_{\omega_{1}}A \nabla z_{1}  \nabla p z_1\; dx\;dt -\int_{0}^T\int_{\omega_{2}}A \nabla z_{2}  \nabla p z_2\; dx\;dt -\int_{\Gamma_T}h p(z_1-z_2)^2\; d\sigma_x\;dt.
\end{array}
\end{equation}
We now estimate the terms on the right hand side of the above expression. To this aim let us fix $i\in \{1,2\}$. We have 
\begin{equation}
 \displaystyle\int_{0}^T\int_{\omega_{i}}|z'_i|^2p\; dx\;dt 
 \leq  \displaystyle\|p\|_{L^\infty(0,T; \Omega)}\int_{0}^T\int_{\omega_{i}}|z'_i|^2 \; dx\;dt
\end{equation}
\noindent and by Young inequality
\begin{equation}
\displaystyle \int_{0}^T\int_{\omega_{i}}z'_ip'z_i\; dx\;dt\leq \frac{1}{2}\displaystyle\|p'\|_{L^\infty(0,T; \Omega)}\left(\int_{0}^T\int_{\omega_{i}}|z_i|^2 \; dx\;dt+\int_{0}^T\int_{\omega_{i}}|z'_i|^2 \; dx\;dt\right).
\end{equation}

\noindent To estimate the other two terms, we apply again Young inequality and hypothesis \eqref{ipoA}ii) to get
\begin{equation}\label{fid4}
\begin{array}{ll}
 \displaystyle \left|\int_{0}^T\int_{\omega_{i}}A \nabla z_{i}  \nabla p z_i\; dx\;dt\right|& \displaystyle\leq \left|\int_{0}^T\int_{\omega_{i}}\beta |\nabla z_{i}||\nabla p||z_i| \;dx\;dt\right| \\
\\ 
& \displaystyle \leq \beta^2\gamma \int_{0}^T\int_{\omega_{i}} p|\nabla z_{i}|^2\; dx\;dt
 \displaystyle+\dfrac{1}{4\gamma}\int_{0}^T\int_{\omega_{i}}\dfrac{|\nabla p|^2}{p}| z_i|^2\; dx\;dt,
\end{array}
\end{equation}
for any $\gamma >0$.
%
By putting the above estimates  into \eqref{fid2} and taking into account \eqref{propp}, we obtain
\begin{equation}\label{fid5}
\begin{array}{ll}
\displaystyle \alpha \left( \int_{0}^{T}\int_{\omega_{1}}|\nabla z_{1}|^2 p\; dx\;dt+\int_{0}^{T}\int_{\omega_{2}}|\nabla z_{2}|^2 p\; dx\;dt
  \right)
  \leq C\left(\int_{0}^T\int_{\omega_{1}}|z'_1|^2\; dx\;dt
 + \int_{0}^T\int_{\omega_{2}}|z'_2|^2\; dx\;dt\right) \\
 \\
 

 \displaystyle \hspace{5cm}+\beta^2\gamma \left(\int_{0}^T\int_{\omega_{1}} |\nabla z_{1}|^2 p \; dx\;dt+ \int_{0}^T\int_{\omega_{2}} |\nabla z_{2}|^2 p\; dx\;dt\right),
\end{array}
\end{equation}
for some constant $C>0$ and for any $\gamma >0$. Thus, choosing $\gamma<\dfrac{\alpha}{\beta^2}$
and by  \eqref{invineq5*} and  \eqref{propp}, we get the desired result.
\end{proof}
\begin{corollary}[equivalence of norms]\label{equivalence}
Let us suppose that $\Omega_2$ is star-shaped with respect to a point $x^0\in \Omega_2$ satisfying condition \eqref{condr0}. Assume that $\omega_1$ and $\omega_2$ are neighbourhoods of ${\partial\Omega(x^0)}$ and $\Gamma$ respectively and let $z=(z_1,z_2)$ the solution of problem \eqref{eq2.6h}-\eqref{regip}. Then, there exists $T_0>0$ such that
\begin{equation}\label{invineq5'}
 E(0)\leq  C(T)\left(\int_0^T \int_{\omega_1}\left(\left|z_1'\right|^2+\left|z_1\right|^2\right)\,dx\,dt+ \int_0^T \int_{\omega_2}\left(\left|z_2'\right|^2+\left|z_2\right|^2\right)\,dx\,dt\right) \leq C_3(T) E(0),
\end{equation}
for $T$ large enough so that \eqref{condT} is satisfied.
\end{corollary}
\begin{proof} The proof is an immediate consequence of   \eqref{stima}, \eqref{energyz}, \eqref{cons} and \eqref{invineq5}. \end{proof}
\noindent The above lemma essentially shows the equivalence of the standard norm in $H_\Gamma$ with the norm

\[\left(\int_0^T \int_{\omega_1}\left(\left|z_1'\right|^2+\left|z_1\right|^2\right)\,dx\,dt+ \int_0^T \int_{\omega_2}\left(\left|z_2'\right|^2+\left|z_2\right|^2\right)\,dx\,dt\right)^{1/2} .
\]
It also proves the following uniqueness result:
if $z_i =0$ in $\omega_i\times (0,T)$, then $z_i= 0$ in $\Omega_i\times (0,T)$, for $i=1,2$.

These are the main points to develop the HUM method described in the next section.

\section{HUM and the internal exact controllability result }\label{secHUM}
In this section, by using the Hilbert Uniqueness Method introduced by Lions (see \cite{ki27, ki27*}), we prove the internal exact controllability of system (\ref{eq2.6}) stated in the following theorem. 
\begin{theorem}\label{mainteo}
Assume that $\eqref{ipoA}$ and $\eqref{ipoh}$ hold. Suppose that $\Omega_2$ is star-shaped with respect to a point $x^0\in \Omega_2$ satisfying $R(x^0)<\alpha/(nM)$. Let $\omega_1$ and $\omega_2$ be neighbourhoods of ${\partial\Omega(x^0)}$ and $\Gamma$, respectively. Then, for any given $\left(U^{0},U^{1}\right)$ in $\left(L^{2}\left(\Omega_{1}\right)\times L^{2}\left(\Omega_{2}\right)\right) \times (H_\Gamma)'$, there exist a control $\zeta \in W'$  and a time $T_0>0$ such that the corresponding solution of problem \eqref{eq2.6} satisfies
\begin{equation}\label{excont}
u(T)=u'(T)=0,\end{equation}
for $T$ large enough so that
\begin{equation}\label{T2}
\dfrac{T-T_0}{T}>\dfrac{n R(x^0)M}{ \alpha}.
\end{equation}
\end{theorem} 

\noindent We point out that the exact controllability is achieved in the space $\left(L^{2}\left(\Omega_{1}\right)\times L^{2}\left(\Omega_{2}\right)\right) \times (H_\Gamma)'$ with control in $W'$. In fact, we represent the control $\zeta$ in terms of the solution $z$ of problem \eqref{eq2.6h}-\eqref{regip} with appropriate chosen initial data. The method is constructive, indeed one could develop it as a numerical algorithm. We briefly describe the HUM which essentially relies on the observability estimate, given in Lemma \ref{invineq5**}.\\
Given any $\left(z^{0},z^{1}\right)\in    H_\Gamma \times  \left(L^{2}\left(\Omega_{1}\right)\times L^{2}\left(\Omega_{2}\right)\right)$, let $z$ the solution of problem \eqref{eq2.6h}-\eqref{regip}. Then consider the following adjoint problem
\begin{equation}\label{probtheta}
\left\{
\begin{array}{@{}ll}
L\theta_i \equiv \theta_i'' - \hbox {div} (A(x)\nabla \theta_{i} )=(-z''_i + z_i) \chi_{\omega_i}& \hbox {in } Q_i,\; \hbox{ for } \; i=1,2\\[3pt]
A(x) \nabla \theta_{1}   \nuno =
   -A(x) \nabla \theta_{2}   \ndue & \hbox {on }
\Gamma_T, \\[3pt]
A(x) \nabla \theta_{1}   \nuno
   = -h(x)(\theta_{1}-\theta_{2}) & \hbox {on }
\Gamma_T, \\[3pt]
\theta_{1} = 0 & \hbox {on } \Sigma,\\[3pt]
\theta_{i}(T)=\theta'_{i}(T)
= 0 & \hbox {in }  \Omega_{i} \; \hbox{ for } \; i=1,2,

\end{array}
\right.
\end{equation}
where the solution $\theta=(\theta_1,\theta_2)$ is intended in the sense of transposition. Here $-z''_i \chi_{\omega_i}$, for $i=1,2$ is to be interpreted in a duality sense, namely
\begin{equation}\label{duality}
\langle -z''_i \chi_{\omega_i}, v_i \rangle_{W', W}= \displaystyle\int_{0}^T\int_{\omega_{i}}z'_i  v'_i \; dx\;dt,
\end{equation}
for all $v=(v_1,v_2) \in W$. \\
 Now, if $(U^0, U^1)$ are the initial conditions of problem \eqref{eq2.6} with $\zeta_i=-z_i''+z_i$, then,  by uniqueness, the null controllability problem is solved if $\theta$ satisfies
\begin{equation}\label{citeta}\theta_i(0) = U_i^0, \;\;  \theta'_i(0) = U_i^1
\end{equation}
for $i=1,2$. Thus, the key point is to choose the initial data  $\left(z^{0},z^{1}\right)\in    H_\Gamma \times  L^{2}\left(\Omega) \right) $ so that the above initial conditions for $\theta$ are satisfied. This motivates us to define the linear operator 
\begin{equation}\label{Le}
\Lambda :H_{\Gamma}  \times  \left(L^{2}\left(\Omega_{1}\right)\times L^{2}\left(\Omega_{2}\right)\right) \rightarrow \left(H_{\Gamma}\right)' \times   \left(L^{2}\left(\Omega_{1}\right)\times L^{2}\left(\Omega_{2}\right)\right) 
\end{equation}
as follows
\begin{equation}\label{Le*}
\Lambda \left(z^{0},z^{1}\right)= \left(\theta'(0),-\theta(0)\right).
\end{equation}
Hence the null controllability problem reduces to prove that $\Lambda $ is onto, since then one can solve 
$$\Lambda \left(z^{0},z^{1}\right)=(U^1, - U^0)$$
to obtain suitable initial values $\left(z^{0},z^{1}\right)$ leading  to \eqref{citeta}. In fact, we prove that $\Lambda$ is an isomorphism and then the solution of the above equation is unique. In this direction, we compute
\begin{equation}\label{Le**}
\begin{array}{ll}
\left\langle \Lambda \left(z^{0},z^{1}\right), \left(z^{0},z^{1}\right) \right\rangle &=\left\langle \left(\theta'(0),-\theta(0)\right), \left(z^{0},z^{1}\right) \right\rangle\\\\
&=\displaystyle \left\langle \theta'_{1}(0), z^{0}_{1}  \right\rangle_{V',V}-
  \int_{\ouno} z^{1}_{1}\theta_{1}(0)  dx\\
  \\ &\qquad +\displaystyle \left\langle \theta'_{2}(0), z^{0}_{2}  \right\rangle_{(H^1(\Omega_2))',H^1(\Omega_2)} -
  \int_{\odue} z^{2}_{2}\theta_{2}(0)  dx,
\end{array}
\end{equation}
for every $\left(z^{0},z^{1}\right)\in   H_\Gamma \times  \left( L^{2}\left(\Omega_{1}\right)\times L^{2}\left(\Omega_{2}\right) \right)$.\\

%
%
%
%
%
\noindent By definition of transposition solution, it is easy to see that the right hand side of the above equation satisfies 
\begin{equation}\label{soltras'}
\begin{array}{l}
\displaystyle \left\langle \theta'_{1}(0), z^{0}_{1}  \right\rangle_{V',V} -
  \int_{\ouno} z^{1}_{1}\theta_{1}(0)  dx+\displaystyle \left\langle \theta'_{2}(0), z^{0}_{2}  \right\rangle_{(H^1(\Omega_2))',H^1(\Omega_2)} -
  \int_{\odue} z^{2}_{2}\theta_{2}(0)  dx\\
  \\
 \qquad\qquad \displaystyle= \int_0^T \int_{\omega_1}\left(\left|z_1'\right|^2+\left|z_1\right|^2\right)\,dx\,dt+\int_0^T \int_{\omega_2}\left(\left|z_2'\right|^2+\left|z_2\right|^2\right)\,dx\,dt.\end{array}
\end{equation} 

\noindent Thus, we have
\begin{equation}\label{ef*eps}
\left\langle \Lambda \left(z^{0},z^{1}\right), \left(z^{0},z^{1}\right) \right\rangle = \int_{\omega_1\times(0,T)}\left(\left|z_1'\right|^2+\left|z_1\right|^2\right)\,dx\,dt+\int_0^T \int_{\omega_2}\left(\left|z_2'\right|^2+\left|z_2\right|^2\right)\,dx\,dt.
\end{equation}
In view of the equivalence of the norms stated in Corollary \ref{equivalence}, the above identity shows that $\Lambda$ is an isomorphism between
$ H_\Gamma \times  \left( L^{2}\left(\Omega_{1}\right)\times L^{2}\left(\Omega_{2}\right) \right)$ and $ \left(H_{\Gamma}\right)' \times \left( L^{2}\left(\Omega_{1}\right)\times L^{2}\left(\Omega_{2}\right)\right)$, for $T$ large enough so that \eqref{T2} is satisfied. Hence Theorem \ref{mainteo} holds true with exact control
 $$\zeta \chi_{\omega} =(\zeta_1 \chi_{\omega_1}, \zeta_2 \chi_{\omega_2})=((-z''_1 + z_1) \chi_{\omega_1},(-z''_2 + z_2) \chi_{\omega_2}),$$ 
which is an element of $W'$. 

We point out that, unlike classical cases, the lower bound for the control time $T$ depends not only on the geometry of our domain and on the matrix of coefficients of our problem but also on the coefficient of proportionality of the jump of the solution of problem \eqref{eq2.6} with respect to the conormal derivatives via the constant $h_0$.

\section*{Acknowledgments}
This paper was completed during the visit of the second author at the University of Sannio, Department of Science and Techology, whose warm hospitality and support are gratefully acknowledged. The work was supported by the grant FFABR of MIUR. S.M. and C.P. are members of GNAMPA of INDAM.

\section*{Conflicts of interest}
None.

\section*{Authors’ contributions}
The authors conceived and wrote this article in collaboration and with the same responsibility. All of them read and
approved the final manuscript.


\end{document}